\documentclass[a4paper, notitlepage]{amsart}
\usepackage{amsmath, amssymb, amsthm, graphicx, xspace, mathrsfs, url, hyperref}
\newcommand{\N}{\ensuremath{\mathbf{N}}\xspace}
\newcommand{\Z}{\ensuremath{\mathbb{Z}}\xspace}
\newcommand{\Q}{\ensuremath{\mathbb{Q}}\xspace}

\newcommand{\C}{\ensuremath{\mathbb{C}}\xspace}
\newcommand{\A}{\ensuremath{\mathbb{A}}\xspace}

\newcommand{\cH}{\ensuremath{\mathscr{H}}\xspace}
\newcommand{\PP}{\ensuremath{\mathfrak{P}}\xspace}
\mathchardef\mhyphen="2D

\newcommand{\ai}{\ensuremath{\mathfrak{a}}\xspace}
\newcommand{\M}{\ensuremath{\mathfrak{M}}\xspace}
\newcommand{\n}{\ensuremath{\mathfrak{n}}\xspace}

\newcommand{\nn}{\ensuremath{\mathbf{n}}\xspace}
\newcommand{\kk}{\ensuremath{\mathbf{k}}\xspace}
\newcommand{\vv}{\ensuremath{\mathbf{v}}\xspace}

\newcommand{\p}{\ensuremath{\mathfrak{p}}\xspace}

\renewcommand{\l}{\ensuremath{\mathfrak{l}}\xspace}

\newcommand{\OO}{\ensuremath{\mathscr{O}}\xspace}

\newcommand{\comment}[1]{}

\DeclareMathOperator{\Gal}{Gal}
\DeclareMathOperator{\End}{End}
\DeclareMathOperator{\Hom}{Hom}

\newtheorem{theorem}{Theorem}[section]
\newtheorem{proposition}[theorem]{Proposition}
\newtheorem{corollary}[theorem]{Corollary}
\newtheorem{lemma}[theorem]{Lemma}
\newtheorem{sublemma}[theorem]{Sublemma}
\newtheorem*{theor}{Theorem}
\theoremstyle{definition}
\newtheorem{definition}[theorem]{Definition}

\theoremstyle{remark}

\newtheorem{remark}[theorem]{Remark}
\newcommand{\OC}{\ensuremath{\mathbf{S}^D_{X}(U;r)}\xspace}
\newcommand{\OCY}{\ensuremath{\mathbf{S}^D_{Y}(U;r)}\xspace}

\newcommand{\OCV}{\ensuremath{\mathbf{S}^D_{X}(V;r)}\xspace}

\newcommand{\OCVY}{\ensuremath{\mathbf{S}^D_{Y}(V;r)}\xspace}

\newcommand{\OCd}{\ensuremath{\mathbf{V}^D_{X}(U;r)}\xspace}
\newcommand{\OCdY}{\ensuremath{\mathbf{V}^D_{Y}(U;r)}\xspace}

\newcommand{\OCdV}{\ensuremath{\mathbf{V}^D_{X}(V;r)}\xspace}
\newcommand{\s}{\ensuremath{^{Q}}\xspace}
\newcommand{\UU}{\ensuremath{U_1(\n)\cap U_0(\pi^\alpha)}\xspace}
\newcommand{\VV}{\ensuremath{U_1(\n)\cap U_0(\l\pi^\alpha)}\xspace}
\newcommand{\GL}{\ensuremath{\mathrm{GL}\xspace}}
\newcommand{\SL}{\ensuremath{\mathrm{SL}\xspace}}
\input xy
\xyoption{all}

\begin{document}
\bibliographystyle{abbrv}
\title[Level raising for $p$-adic Hilbert modular forms]{Level raising for $p$-adic Hilbert modular forms}
\author{James Newton}
\email{jjmn2@cam.ac.uk}
\date{\today}
\begin{abstract}
This paper generalises previous work of the author to the setting of overconvergent $p$-adic automorphic forms for a definite quaternion algebra over a totally real field. We prove results which are analogues of classical `level raising' results in the theory of mod $p$ modular forms. Roughly speaking, we show that an overconvergent eigenform whose associated local Galois representation at some auxiliary prime $\l$ is (a twist of) a direct sum of trivial and cyclotomic characters lies in a family of eigenforms whose local Galois representation at $\l$ is generically (a twist of) a ramified extension of trivial by cyclotomic.

We give some explicit examples of $p$-adic automorphic forms to which our results apply, and give a general family of examples whose existence would follow from counterexamples to the Leopoldt conjecture for totally real fields. 

These results also play a technical role in other work of the author on the problem of local--global compatibility at Steinberg places for Hilbert modular forms of partial weight one.
\end{abstract}
\maketitle
\section{Introduction}
In the paper \cite{chicomp}, we proved some results related to the levels of overconvergent $p$-adic automorphic forms associated with definite quaternion algebras over $\mathbb{Q}$ (as defined by Buzzard \cite{Bu1}). The main result is, roughly speaking, that an overconvergent automorphic form whose associated local Galois representation at some auxiliary prime $l$ is (a twist of) a direct sum of trivial and cyclotomic characters, lies in a family of eigenforms whose local Galois representation at $l$ is generically (a twist of) a ramified extension of trivial by cyclotomic. Any classical member of such a family necessarily generates an automorphic representation whose local factor at $l$ is (a twist of) Steinberg, so this can be regarded as an analogue of the level raising results of Ribet \cite{MR804706} and Diamond--Taylor \cite{DT}. Such results were conjectured by Paulin in his work on local--global compatibility \cite{Pa} (see also \cite{NonComp}). 

In this paper we generalise results of \cite{chicomp} to definite quaternion algebras over a totally real field $F$. Something which we would like to emphasize in this paper, which was not made clear in our previous work, is that our results apply to many $p$-adic automorphic forms whose attached Galois representation is \emph{reducible} --- we only have to exclude things which look something like a weight $2$ Eisenstein series. 

We briefly set-up some notation so we can describe a version of our main theorem, Theorem \ref{raise}. Let $D$ be a definite quaternion algebra over a totally real number field $F$. Fix a prime $p$ and let $\n$ be an ideal of $\OO_F$, coprime to $p$ and the discriminant $\delta$ of $D$. Denote by $\mathscr{E}^D(\n)$ the nilreduction of the eigenvariety of tame level $U_1(\n)$\footnote{by which we mean the usual subgroup of $(D\otimes_F{\A_{F,f}})^\times$, see section \ref{notdef}} constructed in \cite[Part III]{Bu2}. Let $\l$ be a prime ideal of $\OO_F$, coprime to $p\delta$. As remarked above, our results do not apply to certain points of $\mathscr{E}^D(\n)$ --- they all have reducible associated Galois representations, and their weights are, up to a shift in central character and twist by a finite order character, of parallel weight two. We call such points `very Eisenstein' --- more precisely, the points correspond to very Eisenstein maximal ideals of Hecke algebras in the sense of definition \ref{eisideal}.

\begin{theor}
Suppose we have a point $\phi \in \mathscr{E}^D(\n)$ which is not very Eisenstein, and with $T_\l^2(\phi)-(\mathbf{N}\l+1)^2S_\l(\phi)=0$. Let the roots of the $\l$-Hecke polynomial\footnote{this is $X^2-T_\l(\phi)X+(\mathbf{N}\l)S_\l(\phi)$} corresponding to $\phi$ be $\alpha$ and $(\mathbf{N}\l)\alpha$ where $\alpha \in \C_p^\times$. Then the image of the point $\phi$ in $\mathscr{E}^D(\n\l)$ lies in a family of $\l$-new points\footnote{these are the points arising from the natural analogue of $\l$-newforms --- more precisely, we mean points lying in the subspace $\mathscr{E}^{D}(\n\l)^{\l\mhyphen\mathrm{new}}$, constructed in Proposition \ref{neweig}}. In particular, all the classical points in this family arise from automorphic representations of $(D\otimes_F{\A_{F}})^\times$ whose factor at $\l$ is a special representation\footnote{i.e. a twist of the Steinberg representation of $\GL_2(F_\l)$}.
\end{theor}

We also give some examples of $p$-adic Hilbert modular forms to which our main theorem applies. In fact, all our examples are $p$-ordinary. The first collection (see \ref{leoeg}) only exists if Leopoldt's conjecture is false for $F$ and $p$, and for these examples it is crucial that our results hold in the reducible case. The second kind of example (see \ref{expeg}) is found by doing some explicit computations with members of a Hida family --- we just discuss the case $F=\Q$ for simplicity. We have also taken the opportunity to include some corrections to \cite{chicomp} at the end of the paper.

In comparison to \cite{chicomp}, we have to generalise our arguments to work over a higher dimensional base (since weight space has dimension greater that $1$), and also take care of any possible Leopoldt defect. These modifications could be avoided if one was happy to just work over, say, the one-dimensional subspace of `parallel weights' in weight space. On the other hand, one of the main technical ingredients in our work is a form of Ihara's lemma (lemma \ref{normtriviald}) which requires a more complicated combinatorial argument in this setting, regardless of the region of weight space one works over.

These results also provide a technical input to some recent work on the question of local--global compatibility at Steinberg places for Hilbert modular forms of partial weight one \cite{lgpreprint}. 

\section{Modules of $p$-adic overconvergent automorphic forms and Ihara's lemma}\label{mainbody}
In this section we will prove the results we need about modules of $p$-adic overconvergent automorphic forms for quaternion algebras over totally real fields.

\subsection{Some notation and definitions}\label{notdef}
We begin by setting out some notation and definitions, following the presentation of \cite[Part III]{Bu2}. Let $p$ be a fixed prime, and fix a totally real field $F$ of degree $g$ over $\Q$, with ring of integers $\OO_F$. Let $K_0$ denote the closure in $\overline{\Q}_p$ (a fixed algebraic closure of $\Q_p$) of the compositum of the images of all the field homomorphisms $F \rightarrow \overline{\Q}_p$. We let $K$ be any complete extension of $K_0$, and denote by $I$ the set of field homomorphisms $F \rightarrow K$. If we fix an isomorphism $\C \cong \overline{\Q}_p$ then we may identify $I$ with the set of real places of $F$. We let $J$ denote the set of primes of $F$ dividing $p$. For $j \in J$ we denote by $F_j$ the completion of $F$ at $j$, with ring of integers $\OO_j$. Set $\OO_p = \OO_F \otimes_\Z \Z_p$ and $F_p = F \otimes_\Q \Q_p$. Then $\OO_p = \prod_{j \in J} \OO_j$ and $F_p = \prod_{j \in J} F_j$. Fix uniformisers $\pi_j$ of each $F_j$ and let $\pi \in F_p$ be the element with $j$th component $\pi_j$ for each $j \in J$. We also use $\pi$ to denote the ideal of $\OO_F$ which is the product of the prime ideals over $p$.

Any $i \in I$ gives a map $F_p \rightarrow K$, factoring through the projection $F_p \rightarrow F_j$ for some $j:=j(i) \in J$. The map $i \mapsto j(i)$ hence induces a natural surjection $I \rightarrow J$. For any set $S$, $(a_j) \in S^J$ and $i \in I$ we write $a_i$ to denote $a_{j(i)}$.

Let $D$ be a totally definite quaternion algebra over $F$ with discriminant $\delta$ prime to $p$. Fix a maximal order $\OO_D$ of $D$ and isomorphisms $\OO_D \otimes_{\OO_F} \OO_{F_v} \cong M_2(\OO_{F_v})$ for all finite places $v$ of $F$ where $D$ splits. Note that these induce isomorphisms $D \otimes_F F_v \cong M_2(F_v) $ for all such $v$. We define $D_f=D\otimes_F\A_{F,f}$, where $\A_{F,f}$ denotes the finite adeles over $F$. Write $Nm$ for the reduced norm map from $D_f^\times$ to $\A_{F,f}^\times$. Note that if $g \in D_f$ we can regard the projection of $g$ to $D\otimes_F F_p$, $g_p$, as an element of $M_2(F_p)$. 

For a $J$-tuple $\alpha \in \Z^J_{\ge 1}$, we let $\mathbb{M}_\alpha$ denote the monoid of matrices \[\begin{pmatrix}
a & b \\ c & d \\ \end{pmatrix} \in M_2(\OO_p)=D_p\] such that $\pi_j^{\alpha_j} | c_j$,  $\pi_j \nmid d_j$ and $a_jd_j-b_jc_j \neq 0$ for each $j \in J$. If $U$ is an open compact subgroup of $D_f^\times$ and $\alpha \in \Z^J_{\ge 1}$ we say that $U$ has \emph{wild level} $\geq \pi^\alpha$ if the projection of $U$ to $D_p^\times$ is contained in $\mathbb{M}_\alpha$.

We will be interested in two key examples of open compact subgroups of $D_f^\times$. For $\n$ an ideal of $\OO_F$ coprime to $\delta$, we define $U_0(\n)$ (respectively $U_1(\n)$) to be the subgroup of $D_f^\times$ given by the product $\prod_v U_v$, where $U_v = (\OO_D\otimes\OO_v)^\times$ for primes $v | \delta$, and $U_v$ is the matrices in $D_v^\times=\GL_2(\OO_v)$ of the form $\begin{pmatrix}
\ast & \ast \\ 0 & \ast
\end{pmatrix}$ (respectively $\begin{pmatrix}
\ast & \ast \\ 0 & 1
\end{pmatrix}$) mod $v^{\mathrm{val}_v(\n)}$ for all other $v$. For $\alpha = (\alpha_j)_{j \in J} \in \Z_{\ge 1}^J$, $U_1(\n)$ has wild level $\geq \pi^\alpha$ if for each $j \in J$ $\pi_j^{\alpha_j}$ divides the ideal generated by $\n$ in $\OO_j$.

Suppose we have $\alpha \in \Z_{\ge 1}^J$, $U$ a compact open subgroup of $D_f^\times$ of wild level $\geq \pi^\alpha$ and $A$ a module over a commutative ring $R$, with an $R$-linear right action of $\mathbb{M}_\alpha$. We define an $R$-module $\mathscr{L}(U,A)$ by
$$\mathscr{L}(U,A)=\{f:D_f^\times \rightarrow A : f(dgu)=f(g)u_p\ \hbox{ for all }d \in D^\times, g\in D_f^\times, u \in U\}$$ where $D^\times$ is embedded diagonally in $D_f^\times$. If we fix a set $\{d_i : 1\leq i \leq r\}$ of double coset representatives for the finite double quotient $D^\times \backslash D_f^\times / U$, and write $\Gamma_i$ for the group $d_i^{-1}D^\times d_i \cap U$, we have an isomorphism $$\mathscr{L}(U,A)\rightarrow \bigoplus_{i=1}^r A^{\Gamma_i},$$ given by sending $f$ to $(f(d_1),f(d_2),\ldots,f(d_r))$. 

For $f : D_f^\times \rightarrow A$, $x\in D_f^\times$ with $x_p \in \mathbb{M}_\alpha$, we define $f|x : D_f^\times \rightarrow A$ by $(f|x)(g)=f(gx^{-1})x_p$. Note that we can now also write $$\mathscr{L}(U,A)=\{f:D^\times \backslash D_f^\times \rightarrow A : f|u=f \hbox{ for all }u \in U\}.$$

We can define double coset operators on the spaces $\mathscr{L}(U,A)$. If $U$, $V$ are two compact open subgroups of $D_f^\times$ of wild level $\geq \pi^\alpha$, and $A$ is as above, then for $\eta\in D_f^\times$ with $\eta_p \in \mathbb{M}_\alpha$ we may define an $R$-module map $[U\eta V]:\mathscr{L}(U,A) \rightarrow \mathscr{L}(V,A)$ as follows: we decompose $U\eta V$ into a finite union of right cosets $\coprod_i U x_i$ and define $$f|[U\eta V] = \sum_i f|x_i.$$

\subsection{Overconvergent automorphic forms}\label{ocforms}
We let $\mathscr{W}$ denote the weight space defined in \cite[pg. 65]{Bu2} --- it is a rigid analytic space over $K$ with dimension $g+1+\mathfrak{d}$, where $\mathfrak{d}$ is the defect in Leopoldt's conjecture for $(F,p)$. Affinoid $K$-spaces $X$ equipped with a map to $\mathscr{W}$ correspond bijectively with continuous group homomorphism $\kappa : \OO_p^\times \times \OO_p^\times \rightarrow \OO(X)^\times$, such that the kernel of $\kappa$ contains a subgroup of $\OO_F^\times$ of finite index (where $\OO_F^\times$ is embedded in $\OO_p^\times \times \OO_p^\times$ via the map $\gamma \mapsto (\gamma,\gamma^2)$). 

We write $\kappa = (\nn,\vv)$ where $\nn:\OO_p^\times \rightarrow \OO(X)^\times$ is the composition of the map $\gamma \mapsto (\gamma,1)$ with $\kappa$, and $\vv:\OO_p^\times \rightarrow \OO(X)^\times$ is defined similarly with respect to the second component of $\OO_p^\times \times \OO_p^\times$.

We define locally algebraic points in the weight space $\mathscr{W}$ following \cite[\S 11]{Bu2} (although we use a different notation for weights of Hilbert modular forms). Suppose we have $\kk \in \Z_{\ge 2}^I$, $w \in \Z$ such that the components $k_i$ of $\kk$ are all congruent to $w$ mod $2$, and $\varepsilon : \OO_p^\times \rightarrow K'^\times$ a finite order character (with $K'/K$ a finite extension). Then we denote by $(\kk,w,\varepsilon)\in \mathscr{W}(K')$ the point corresponding to the map \begin{align*}\kappa:\OO_p^\times \times \OO_p^\times &\rightarrow K'^\times\\(a,b) &\mapsto \varepsilon(a)\prod_{i\in I} a_i^{k_i-2}b_i^{\frac{w-k_i}{2}}.\end{align*} In the above $a_i$ denotes the image in $K$ of $a_{j(i)}\in F_j$ under the embedding $i$ (and similarly for $b_i$). 

Let $\mathscr{N}_K$ denote the set $\{|x|:x\in\overline{K}^\times,|x|\le 1\}$ and let $\mathscr{N}_K^\times$ denote $\mathscr{N}_K\backslash \{1\}$. 

For $r \in (\mathscr{N}_K)^J$ we define $\mathbb{B}_{r}$ as in \cite[pp. 64-65]{Bu2} --- it is a rigid analytic subvariety of the unit $g$-polydisc over $K$. For all complete extensions $K'/K$ we have $$\mathbb{B}_{r}(K')=\{z \in K'^I: \exists y\in \OO_p \hbox{ such that for all } i, |z_i-y_i|\leq r_i\}.$$
Similarly if $r \in (\mathscr{N}_K^\times)^J$ we have $\mathbb{B}_{r}^\times$ satisfying $$\mathbb{B}_{r}^\times(K')=\{z \in K'^I: \exists y\in \OO_p^\times \hbox{ such that for all } i, |z_i-y_i|\leq r_i\}.$$

Supposing we have a map of rigid spaces $X \rightarrow \mathscr{W}$, where $X$ is a reduced $K$-affinoid.  We have an associated continuous group homomorphism $\kappa = (n,v) : \OO_p^\times \times \OO_p^\times \rightarrow \OO(X)^\times$. The map $n$ factors as a product (over $j\in J$) of maps $n_j: \OO_j^\times \rightarrow \OO(X)^\times$. Let $r \in (\mathscr{N}^\times_K)^J$. We say that $\kappa$ is $r$-analytic if for each $j \in J$ the character $n_j$ is induced by a morphism of rigid analytic varieties $n_j: \mathbb{B}_{r}^\times \times X \rightarrow \mathbb{G}_m.$

Now for $r \in (\mathscr{N}_K)^J$ we define $\mathscr{A}_{X,r}$ to be the $K$-Banach algebra $\OO(\mathbb{B}_{r} \times X)$. We endow $\mathscr{A}_{X,r}$ with the supremum norm. Suppose we have $\alpha \in \Z_{\ge 1}^J$. We let $r|\pi^\alpha|$ denote the element of $(\mathscr{N}^\times_K)^J$ with $j$ component equal to $r_j|\pi_j^{\alpha_j}|$. If $\kappa$ is $r|\pi^\alpha|$-analytic then we can define a (continuous, norm-decreasing) right action of $\mathbb{M}_\alpha$ on $\mathscr{A}_{X,r}$ as in \cite[pp. 71-72]{Bu2}. First we extend $v$ to a map $v:F_p^\times \rightarrow \OO(X)^\times$ by setting $v(\pi_j)=1$ for all $j\in J$. The formula `on points' for the $\mathbb{M}_\alpha$-action is that if $f \in \mathscr{A}_{X,r}$, $\gamma = \begin{pmatrix}
a&b\\
c&d
\end{pmatrix} \in \mathbb{M}_\alpha ,$
$$(f\cdot \gamma)(z,x)= n(cz+d,x)(v(\det(\gamma))(x))f(\frac{az+b}{cz+d},x),$$ where $x \in X(K')$ and $z \in \mathbb{B}_{r}(K').$ It follows from \cite[Proposition 8.3]{Bu2} that for fixed $(\kappa,r)$ as above, there exists an $\alpha$ such that $\kappa$ is $r|\pi^\alpha|$-analytic. 

We now define spaces of overconvergent automorphic forms, as in \cite[Part III]{Bu2}. 
\begin{definition} \label{ocformdef}Let $X$ be a reduced affinoid over $K$ and let $X\rightarrow \mathscr{W}$ be a morphism of rigid spaces, inducing $\kappa:\OO_p^\times \times \OO_p^\times \rightarrow \OO(X)^\times$. If we have $r \in (\mathscr{N}_K)^J$ and $\alpha \in \Z_{\ge 1}^J$ such that $\kappa$ is $r|\pi^\alpha|$-analytic, and $U$ a compact open subgroup of $D_f^\times$ of wild level $\geq \pi^{\alpha}$, then define the space of $r$-overconvergent automorphic forms of weight $\kappa$ (or `weight $X$') and level $U$ to be the $\OO(X)$-module
$$\mathbf{S}^D_{X}(U;r):=\mathscr{L}(U,\mathscr{A}_{X,r}).$$ 
\end{definition}

If we endow $\mathbf{S}^D_{X}(U;r)$ with the norm $|f|=\max_{g\in D_f^\times}|f(g)|$, then the isomorphism \begin{equation}\mathbf{S}^D_{X}(U;r)\cong \bigoplus_{i=1}^r \mathscr{A}_{X,r}^{\Gamma_i}\label{iso}\end{equation} induced by fixing double coset representatives $d_i$ is norm preserving. As discussed in \cite[pp 68--69]{Bu2}, the groups $\Gamma_i$ act on $\mathscr{A}_{X,r}$ via finite quotients, and $\mathscr{A}_{X,r}$ is an ONable Banach $\OO(X)$-module (it is the base change to $\OO(X)$ of $\OO(\mathbb{B}_{r})$, and all Banach spaces over a discretely valued field are ONable), so $\mathbf{S}^D_{X}(U;r)$ is a Banach $\OO(X)$-module satisfying property $(Pr)$ of \cite{Bu2}.

\begin{lemma}\label{bc}
Suppose $X$, $\kappa$, $r$ and $\alpha$ are as in Definition \ref{ocformdef}. Then for any reduced $K$-affinoid $Y \rightarrow X$ the natural map $\mathscr{A}_{X,r} \rightarrow \mathscr{A}_{Y,r}$ induces an isomorphism of Banach modules \[\mathbf{S}^D_{X}(U;r)\widehat{\otimes}_{\OO(X)}\OO(Y) \cong \mathbf{S}^D_{Y}(U;r).\]
\end{lemma}
\begin{proof}
It is enough to check that the natural map \[\mathscr{A}_{X,r}^{\Gamma_i}\widehat{\otimes}_{\OO(X)}\OO(Y) \rightarrow (\mathscr{A}_{X,r}\widehat{\otimes}_{\OO(X)}\OO(Y))^{\Gamma_i} = \mathscr{A}_{Y,r}^{\Gamma_i}\] is an isomorphism. Denote by $e_X$ the idempotent projection onto the $\Gamma_i$-invariants, so that $\mathscr{A}_{X,r}^{\Gamma_i} = e_X\mathscr{A}_{X,r}$ \[\mathscr{A}_{X,r} = e_X\mathscr{A}_{X,r}\oplus (1-e_X)\mathscr{A}_{X,r},\] hence \[\mathscr{A}_{Y,r} = \mathscr{A}_{X,r}\widehat{\otimes}_{\OO(X)}\OO(Y) = e_X\mathscr{A}_{X,r}\widehat{\otimes}_{\OO(X)}\OO(Y)\oplus (1-e_X)\mathscr{A}_{X,r}\widehat{\otimes}_{\OO(X)}\OO(Y).\] We also have \[\mathscr{A}_{Y,r} = e_Y\mathscr{A}_{Y,r}\oplus (1-e_Y)\mathscr{A}_{Y,r}\] and it is clear that $e_X\mathscr{A}_{X,r}\widehat{\otimes}_{\OO(X)}\OO(Y)\subset e_Y\mathscr{A}_{Y,r}$ and $(1-e_X)\mathscr{A}_{X,r}\widehat{\otimes}_{\OO(X)}\OO(Y) \subset (1-e_Y)\mathscr{A}_{Y,r}$, so we deduce that both these inclusions are equalities, and in particular $e_X\mathscr{A}_{X,r}\widehat{\otimes}_{\OO(X)}\OO(Y)= e_Y\mathscr{A}_{Y,r}$.
\end{proof}

There are maps from spaces of classical automorphic forms into our spaces of overconvergent automorphic forms (with classical weight), which we now describe.  Let $U_0$ be a compact open subgroup of $D_f^\times$ of the form $U' \times \GL_2(\OO_p)$, and choose $\alpha \in \Z_{\ge 1}^J$. Choose a character $\varepsilon : \Delta:=U_0\cap U_0(\pi^\alpha)/U_0 \cap U_1(\pi^\alpha) \rightarrow K'^\times$ for $K'/K$ finite. Suppose we have a classical weight $(k,w,\varepsilon)$, with $(k,w)$ corresponding to suitable $(n,v)$. Then as in \cite[\S 9]{Bu2} we can define a finite dimensional $K$-vector space $L_{n,v}$ with a $K$-linear right action of $\mathbb{M}_\alpha$, such that $\mathscr{L}(U_0\cap U_1(\pi^\alpha),L_{n,v})$ is isomorphic (after tensoring with $\C$) to a classical space of automorphic forms for $D$. Denote by $\mathscr{L}(U_0\cap U_1(\pi^\alpha),L_{n,v})(\varepsilon)$ the $\varepsilon$-eigenspace (over $K'$) for the left action of $\Delta$, induced by $f \mapsto f|u^{-1}$. As described in \cite[\S 9]{Bu2}, for any $r\in (\mathscr{N}_K)^J$ there is an embedding (equivariant with respect to the Hecke operators defined below) $$\iota: \mathscr{L}(U_0\cap U_1(\pi^\alpha),L_{n,v})(\varepsilon) \rightarrow \mathbf{S}^D_{(k,w,\varepsilon)}(U_0\cap U_0(\pi^\alpha);r).$$

\subsection{An explicit description of the action of special elements of $\mathbb{M}_\alpha$ on $\mathscr{A}_{X,r}$}\label{explicit}
Later in this paper we will have to perform a rather explicit calculation involving $\mathscr{A}_{X,r}$ and the action of a certain element of $\mathbb{M}_\alpha$. Suppose we have $r \in (\mathscr{N}_K)^J$ such that for every $j \in J$, $r_j = |x_j|$ for some $x_j \in K$ (we can always enlarge $K$ to ensure this). Let $I_r$ denote the ideal in $\OO_p$ whose elements are $y\in \OO_p$ such that $|y_j|\le r_j$ for each $j \in J$. Let $S\subset \OO_p$ denote a set of representatives for $\OO_p/I_r$. We write $\mathbb{B}_{r}$ as a disjoint union of connected rigid spaces $$\mathbb{B}_{r}= \coprod_{s \in S} \mathbb{B}^s_{r}$$ where we have
$$ \mathbb{B}^s_{r}(K') = \{z \in K'^I: \hbox{ such that for all } i, |z_i-s_i|\leq r_i\}.$$
Let $\mathbf{1}\in (\mathscr{N}_K)^J$ denote the element with $1$ in every component. There is an isomorphism 
$$\nu_r : S \times\mathbb{B}_{\mathbf{1}} = \coprod_{s \in S}\mathbb{B}_{\mathbf{1}} \rightarrow \mathbb{B}_{r}$$ which maps $(s,z)$ on the left hand side to $(z_ix_i+y_i)_{i\in I}\in \mathbb{B}^s_{r}.$ This induces an isomorphism $\nu_r^*:\mathscr{A}_{X,r} \rightarrow \prod_{s \in S}\mathscr{A}_{X,\mathbf{1}}$. We choose $N$ a positive integer, sufficiently large so that $|\pi_j|^N < r_j$ for all $j \in J$, and compute how $\nu_r$ intertwines the action of $\gamma=\begin{pmatrix}
1 & \pi^N\\0 & 1
\end{pmatrix}.$ Note that $\gamma \in \mathbb{M}_{\alpha}$ for all $\alpha$.
\begin{lemma}\label{simplecalc}
For $f\in \mathscr{A}_{X,r}$, $x \in X$ and $(s,z)$ in $S \times\mathbb{B}_{\mathbf{1}}$ we have $$\nu_r^*(f\cdot\gamma)(x,(s,z))=\nu_r^*(f)(x,(s,(z_i+\pi_i^N/x_i)_{i\in I})).$$
\end{lemma}
\begin{proof}This is a simple calculation. Note that the assumption on $N$ implies that $(z_ix_i+y_i+\pi^N)_{i\in I}$ remains in the connected component $\mathbb{B}^s_{r}$\end{proof}
\subsection{Dual modules}
Suppose $A$ is a Banach algebra. Given a Banach $A$-module $\mathbf{M}$ we define the \emph{dual} $\mathbf{M}^*$ to be the Banach $A$-module of continuous $A$-module morphisms from $\mathbf{M}$ to $A$, with the usual operator norm. We denote the $\OO(X)$-module $\mathscr{A}_{X,r}^*$ by $\mathscr{D}_{X,r}$.

If the map $\kappa$ corresponding to $X$ is $r|\pi^\alpha|$-analytic, then $\mathbb{M}_\alpha$ acts continuously on $\mathscr{A}_{X,r}$, so $\mathscr{D}_{X,r}$ has an $\OO(X)$-linear right action of the monoid $\mathbb{M}_\alpha^{-1}$ given by $(f\cdot m^{-1})(x):=f(x\cdot m)$, for $f\in \mathscr{D}_{X,r}$, $x \in \mathscr{A}_{X,r}$ and $m \in \mathbb{M}_\alpha$. If $U$ is as in Definition~\ref{ocformdef} then its projection to $\mathrm{GL}_2(F_p)$ is contained in $\mathbb{M}_\alpha \cap \mathbb{M}_\alpha^{-1}$, so it acts on  $\mathscr{D}_{X,r}$. This allows us to make the following definition:

\begin{definition}
For $X$, $\kappa$, $r$, $\alpha$ and $U$ as above, we define the space of \emph{dual} \\$r$-overconvergent automorphic forms of weight $X$ and level $U$ to be the $\OO(X)$-module
$$\mathbf{V}^D_{X}(U;r):=\mathscr{L}(U,\mathscr{D}_{X,r}).$$
\end{definition}

As in \ref{ocforms}, we have a norm preserving isomorphism \begin{equation}\mathbf{V}^D_{X}(U;r)\cong \bigoplus_{i=1}^r \mathscr{D}_{X,r}^{\Gamma_i}.\label{dualiso}\end{equation} Thus $\mathbf{V}^D_{X}(U;r)$ is a Banach $\OO(X)$-module, not necessarily satisfying the property $(Pr)$ (since $\mathscr{D}_{X,r}$ is not necessarily ONable). 

If $U$, $V$ are two compact open subgroups of $D_f^\times$ of wild level $\geq \pi^\alpha$, then for $\eta\in D_f^\times$ with $\eta_p \in \mathbb{M}_\alpha^{-1}$ we get double coset operators $[U\eta V]:\mathbf{V}^D_{X}(U;r) \rightarrow \mathbf{V}^D_{X}(V;r)$.

\subsection{Hecke operators}
For an ideal $\ai$ of $\OO_F$, we define the \emph{Hecke algebra away from} $\ai$, $\mathbb{T}^{(\ai)}$,  to be the free commutative $\OO(X)$-algebra generated by symbols $T_v, S_v$ for finite places $v$ of $F$ prime to $\ai$. If $\delta p$ divides $\ai$ then we can define the usual action of $\mathbb{T}^{(\ai)}$ by double coset operators on $\OC$: for $v\nmid \delta$ define $\varpi_v\in\A_f$ to be the finite adele which is $\pi_v$ at $\pi$ and $1$ at the other places. Abusing notation slightly, we also write $\varpi_v$ for the element of $D_f^\times$ which is $\begin{pmatrix}
\pi_v & 0\\0&\pi_v
\end{pmatrix}$ at $v$ and the identity elsewhere. Similarly set $\eta_v$ to be the element of $D_f^\times$ which is $\begin{pmatrix}
\pi_v & 0\\0&1
\end{pmatrix}$ at $v$ and the identity elsewhere. On $\OC$ we let $T_v$ act by $[U\eta_v U]$ and $S_v$ by $[U\varpi_v U]$.
Similarly on $\OCd$ we define $T_v$ to act by $[U\eta_v^{-1}U]$ and $S_v$ by $[U\varpi_v^{-1} U]$. 
We use different notation for Hecke operators at primes above $p$. If $p$ factors in $F$ as $\prod_{j \in J} \p_j^{e_j}$ then we let $U_j$ denote the Hecke operator $[U\eta_{\p_j}U]$ acting on $\OC$ and let $U_\pi$ denote $\prod_{j \in J} U_j$. It follows from \cite[Lemma 12.2]{Bu2} that $U_\pi$ is a norm-decreasing compact endomorphism of $\OC$.

As remarked above, for a classical weight $(k,w,\varepsilon)$ (with associated $n,v$), there is an embedding $$\iota: \mathscr{L}(U_0\cap U_1(\pi^\alpha),L_{n,v})(\varepsilon) \rightarrow \mathbf{S}^D_{(k,w,\varepsilon)}(U_0\cap U_0(\pi^\alpha);\mathbf{1}).$$

The well-known classicality criterion in this setting (for example, it is a special case of \cite[Theorem 3.9.6]{L}) tells us that generalised eigenvectors for the operator $U_\pi$ on the right hand side of the above map, with eigenvalue of sufficiently small slope, lie in the image of $\iota$. We do not need to recall the precise criterion here, it will be sufficient to record its consequences for Zariski density of classical points in the eigenvariety.

\begin{definition}\label{eigdef}
Let $\mathscr{E}^D(\n)$ be the reduction of the eigenvariety of tame level $U_1(\n)$ constructed (as in \cite{Bu2}) from the spaces of overconvergent forms defined in Definition \ref{ocformdef}, the compact operator $U_{\pi}$ and the Hecke algebra $\mathbb{T}^{(\delta p\n)}$. It comes equipped with a map $\mathscr{E}^D(\n) \rightarrow \mathscr{W}$. We say that a point of $\mathscr{E}^D(\n)$ is \emph{classical} if it corresponds to a system of Hecke eigenvalues arising in the image of $\iota$, with the caveat that we exclude points arising from $f=\iota(f_0)$ where $f_0 : D_f^\times \rightarrow L_{n,v}$ factors through the reduced norm $Nm:D_f^\times \rightarrow \A_{F,f}^\times$ (so $n=0$), and denote the set of classical points by $\mathscr{Z}^{cl}$. 

We say that a point $x$ of $\mathscr{E}^D(\n)$ is \emph{essentially classical} if there exists a $p$-adic character $\psi$ of $\Gal(\overline{F}/F)$, unramified above places not dividing $p\delta\n$, such that the twist by $\psi$ of the $p$-adic Galois representation associated to $x$ is isomorphic to the Galois representation associated to some cuspidal Hilbert modular form.
\end{definition}

The reason for our terminology is that the classical points correspond to cuspidal Hilbert modular eigenforms of cohomological weight. Note that we could formulate the definition of essentially classical purely in terms of Hecke eigenvalues, but it seems simplest to give the Galois theoretic description. We have to use the notion of essentially classical to avoid assuming Leopoldt's conjecture for $(F,p)$: the Zariski closure of the classical weights in $\mathscr{W}$ has dimension $1+g$ whilst $\mathscr{W}$ has dimension $1+g+\mathfrak{d}$. See also \cite{cheninf}, before the statement of Theorem $5.9$, for the same notion.

The following is a consequence of \cite[Part III]{Bu2} (see also \cite{cheninf}, following the statement of Theorem 5.9).

\begin{theorem}\label{Zardens}
The eigenvariety $\mathscr{E}^D(\n)$ is equidimensional of dimension $1+g+\mathfrak{d}$. The map $\mathscr{E}^D(\n) \rightarrow \mathscr{W}$ is locally finite, and $\mathscr{Z}^{ecl}$ is a Zariski dense subset of $\mathscr{E}^D(\n)$.
\end{theorem}

\subsection{A pairing}
In this section $X$, $\kappa$, $r$, $\alpha$ and $U$ will be as in Definition~\ref{ocformdef}. We will denote by $V$ another compact open subgroup of wild level $\geq \pi^\alpha$. We fix double coset representatives $\{d_i : 1\leq i \leq r\}$ for the double quotient $D^\times \backslash D_f^\times / U$. 

\begin{definition}For $x \in D_f^\times$ we denote by $\gamma_U(x)$ the rational number \[\frac{[x^{-1}D^\times x \cap U:\OO_F^{\times,+} \cap U]}{[\OO_F^{\times,+}:\OO_F^{\times,+} \cap U]}\] where $\OO_F^{\times,+}$ denotes the group of totally positive units.
\end{definition}
\begin{lemma}\label{meascalc}
For $u \in U$, $d \in D^\times$ and $x,g \in D_f^\times$ we have \begin{itemize}\item $\gamma_U(dxu)=\gamma_U(x)$
\item $\gamma_U(x)=\gamma_{gUg^{-1}}(xg^{-1})$
\item if $U' \subset U$ we have \[\frac{\gamma_U(x)}{\gamma_{U'}(x)}=[x^{-1}D^\times x \cap U:x^{-1}D^\times x \cap U']\]\end{itemize} 
\end{lemma}
\begin{proof}
For the first item, observe that conjugation by $u$ maps the coset space $x^{-1}D^\times x \cap U/\OO_F^{\times,+} \cap U$ bijectively to $(dxu)^{-1}D^\times dxu \cap U/\OO_F^{\times,+} \cap U.$

Similarly the second item follows from considering conjugation by $g$.

Finally, we have \begin{align*}\frac{\gamma_U(x)}{\gamma_{U'}(x)}&=\frac{[x^{-1}D^\times x \cap U:\OO_F^{\times,+} \cap U]}{[\OO_F^{\times,+}:\OO_F^{\times,+} \cap U]}\frac{[\OO_F^{\times,+}:\OO_F^{\times,+} \cap U']}{[x^{-1}D^\times x \cap U':\OO_F^{\times,+} \cap U']}\\&=[\OO_F^{\times,+} \cap U:\OO_F^{\times,+} \cap U']\frac{[x^{-1}D^\times x \cap U:\OO_F^{\times,+} \cap U]}{[x^{-1}D^\times x \cap U':\OO_F^{\times,+} \cap U']}\\ &=\frac{[x^{-1}D^\times x \cap U:\OO_F^{\times,+} \cap U']}{[x^{-1}D^\times x \cap U':\OO_F^{\times,+} \cap U']}=[x^{-1}D^\times x \cap U:x^{-1}D^\times x \cap U'].\end{align*}
\end{proof}
\begin{remark}
The quantity $\gamma_U(x)$ is the natural measure of the size of the group $x^{-1}D^\times x \cap U$, which is commensurable with $\OO_F^{\times,+}$. If $F=\Q$ we have $\OO_F^{\times,+}$ trivial and $\gamma_U(x)$ is the cardinality of the finite group $x^{-1}D^\times x \cap U$.
\end{remark}

We define an $\OO(X)$-bilinear pairing between the spaces $\OC$ and $\OCd$ by
\begin{equation}\label{pair}
\langle f,\lambda \rangle:=\sum_{i=1}^r \gamma_U(d_i)^{-1}\langle f(d_i), \lambda(d_i)\rangle,\end{equation}
where $f \in \OC$, $\lambda \in \OCd$ and on the right hand side of the above definition $\langle\cdot , \cdot \rangle$ denotes the pairing between $\mathscr{A}_{X,r}$ and $\mathscr{D}_{X,r}$ given by evaluation.

This pairing is independent of the choice of the double coset representatives $d_i$, since for every $d \in D^\times$, $g \in D^\times_f$, $u \in U$, $f\in \OC$ and $\lambda \in \OCd$ we have $$\langle f(dgu),\lambda(dgu)\rangle=\langle f(g)u_p,\lambda(g)u_p \rangle=\langle f(g)u_p u_p^{-1},\lambda(g)\rangle=\langle f(g),\lambda(g)\rangle$$ and $\gamma_U(g)=\gamma_U(dgu)$. Combining this observation with the isomorphisms (\ref{iso}) and (\ref{dualiso}) we see that our pairing identifies $\OCd$ with $\OC^*$.

The following proposition is a mild generalisation of \cite[Proposition 3]{chicomp}. We give the proof here since the proof in loc. cit. omits a factor. 
\begin{proposition}\label{prop:pairinghecke}
Let $f \in \OC$ and let $\lambda\in \mathbf{V}^D_X(V;r)$. Let $g \in D^\times_f$ with $g_p\in\mathbb{M}_\alpha$. Then $$\langle f|[U g V],\lambda\rangle=\langle f,\lambda|[Vg^{-1}U]\rangle.$$
\end{proposition}
\begin{proof}
We have $$f|[Ug V]=\sum_{v\in (g^{-1} U\eta)\cap V \backslash V}f|(g v),$$ hence \begin{eqnarray*}\langle f|[Ug V],\lambda\rangle&=&\sum_{d\in D^\times\backslash D^\times_f/V} \gamma_V(d)^{-1}\langle f|[Ug V](d),\lambda(d)\rangle\\
&=&\sum_{d\in D^\times\backslash D^\times_f/V}\sum_{v\in (g^{-1} Ug)\cap V \backslash V} \gamma_V(d)^{-1}\langle f|(g v)(d),\lambda(d)\rangle\\
&=&\sum_{d\in D^\times\backslash D^\times_f/V}\sum_{v\in (g^{-1} Ug)\cap V \backslash V} \gamma_V(dv^{-1})^{-1}\langle f(dv^{-1}g^{-1})\cdot g_p v_p,\lambda(d)\rangle\\
&=&\sum_{x\in D^\times\backslash D^\times_f/(g^{-1} Ug)\cap V}\gamma_{(g^{-1} Ug)\cap V}(x)^{-1}\langle f(xg^{-1}),\lambda(x)\cdot g_p^{-1}\rangle\\
&=&\sum_{y\in D^\times\backslash D^\times_f/U\cap(g V g^{-1})}\gamma_{U\cap (gVg^{-1})}(y)^{-1}\langle f(y),\lambda(yg)\cdot g_p^{-1}\rangle\\
&=&\langle f,\lambda|[Vg^{-1}U]\rangle
\end{eqnarray*}where we pass from the third line to the fourth line by substituting $x = dv^{-1}$ and observing that under the map \begin{align*}P:D^\times\backslash D^\times_f/V \times (g^{-1} Ug)\cap V \backslash V &\rightarrow D^\times\backslash D^\times_f/(g^{-1} Ug)\cap V\\(d,v)&\mapsto dv^{-1}\end{align*} a fibre $P^{-1}(x)$ has cardinality $\gamma_V(x)/\gamma_{(g^{-1} Ug)\cap V}(x)$ (applying the final part of lemma \ref{meascalc}). For the next line we substitute $y = xg^{-1}$ and apply the second part of lemma \ref{meascalc}. The final line follows by similar calculations to the first $5$ lines.
\end{proof}
\subsection{Slope decompositions}
Let $X \rightarrow \mathscr{W}$ be a reduced affinoid and denote by $F(T)$ the characteristic power series of $U_\pi$ acting on $\OC$. Suppose we have a factorisation $F(T)=Q(T)S(T)$ with $Q(T)$ a polynomial in $1 + T\OO(X)[T]$ of degree $n$ with unit leading coefficient, such that $Q$ and $S$ are coprime. Denote by $Q*(T)$ the polynomial $T^nQ(T^{-1})$. Then, applying \cite[Theorem 3.3]{Bu2}, we have a $U_\pi$ stable decomposition $$\OC=\OC^{Q}\oplus \mathbf{N},$$ where $Q^*(U_\pi)$ is zero on $\OC^{Q}$ and invertible on $\mathbf{N}$. Its formation commutes with base change: for a reduced affinoid $Y \rightarrow X$ we write $Q_Y$ for the polynomial obtained from $Q$ by applying the map $\OO(X)\rightarrow \OO(Y)$ to its coeffients. Then we have $$\OCY=\OCY^{Q_Y}\oplus \mathbf{N}_Y,$$ and we can identify $\OCY^{Q_Y}$ with $\OC^{Q}\otimes_{\OO(X)}\OO(Y)$.

The space $\OC^{Q}$ is a projective finitely generated $\OO(X)$-module. The decomposition into $\OC^Q$ and $\mathbf{N}$ is stable under the action of $\mathbb{T}^{(\delta p)}$, since the $T_v$ and $S_v$ operators for $v \nmid p$ commute with $U_\pi$. 

Recall that we have used the pairing (\ref{pair}) to identify $\OCd$ and the $\OO(X)$-dual of $\OC$. We define the submodule $\OCd^{Q}$ of $\OCd$ to be those maps from $\OC$ to $\OO(X)$ which are $0$ on $\mathbf{N}$. This space is also stable under the action of $\mathbb{T}^{(\delta p)}$ and is naturally isomorphic to the $\OO(X)$-dual of $\OC^{Q}$. The pairing (\ref{pair}) is perfect when restricted to $\OC^{Q}\times \OCd^{Q}$, by \cite[Lemma 4]{chicomp}. 

\begin{lemma}\label{bcimm}
Suppose $Y\hookrightarrow X$ is a closed immersion of reduced affinoids. Then we can identify $\OCdY^{Q}$ with $\OCd^{Q}\otimes_{\OO(X)}\OO(Y)$. In particular, they are isomorphic as Hecke modules.
\end{lemma}
\begin{proof}
Denote the finitely generated projective $\OO(X)$-module $\OC^{Q}$ by $M$ and denote by $I$ the kernel of $\OO(X)\twoheadrightarrow\OO(Y)$. 

Since we have already identified $\OCd^{Q}$ with the dual of $M$, to prove the lemma we need to show that we can identify $\Hom_{\OO(Y)}(M\otimes_{\OO(X)}\OO(Y),\OO(Y))$ and $\Hom_{\OO(X)}(M,\OO(X))\otimes_{\OO(X)} \OO(Y)$.

We apply the exact functor $\Hom_{\OO(X)}(M,-)$ to the short exact sequence of $\OO(X)$-modules \[0\rightarrow I \rightarrow \OO(X) \rightarrow \OO(Y)\rightarrow 0\] to obtain a short exact sequence \[0 \rightarrow\Hom_{\OO(X)}(M,I) \rightarrow \Hom_{\OO(X)}(M,\OO(X))\rightarrow \Hom_{\OO(X)}(M,\OO(Y))\rightarrow 0.\] Since $M$ is finitely generated, the quotient $\Hom_{\OO(X)}(M,\OO(X))/\Hom_{\OO(X)}(M,I)$ can be identified with  $\Hom_{\OO(X)}(M,\OO(X))\otimes_{\OO(X)} \OO(Y)$.

Since $\Hom_{\OO(X)}(M,\OO(Y)) = \Hom_{\OO(Y)}(M\otimes_{\OO(X)}\OO(Y),\OO(Y))$, we deduce the lemma.
\end{proof}

\subsection{Old and new}\label{oldnew}
Fix a non-zero ideal $\n$ of $\OO_F$ (the tame level) coprime to $p$ and fix another finite place $\l \nmid \n p\delta$. Let $X\rightarrow \mathscr{W}$ be a reduced affinoid. We assume $X$ is such that the associated character $\kappa$ is $r|\pi^{\alpha}|$ analytic, for some $\alpha \in \Z_{\ge 1}^J$ and $r \in (\mathscr{N}_K)^J$. Set $U=\UU$, $V=\VV$. To simplify notation we set \begin{eqnarray*}L&:=&\OC^{Q}\\L^*&:=&\OCd\s \\M&:=&\OCV\s\\M^*&:=&\OCdV\s.\end{eqnarray*} 

We define a map $i:L \times L \rightarrow M$ by
$$i(f,g):=f|[U 1 V]+g|[U \eta_\l V].$$ Since the map $i$ is defined by double coset operators with trivial component at all places dividing $p$ it commutes with $U_\pi$ and thus gives a well defined map between these spaces of $Q$-bounded forms.

Regarding $f$ and $g$ as functions on $D^\times_f$ we have $f|[U 1 V]=f$, $g|[U\eta_\l V]=g|\eta_\l$. The image of $i$ inside $M$ will be referred to as the space of \emph{oldforms}.

We also define a map $i^\dagger:M \rightarrow L \times L$ by
$$i^\dagger (f):=(f|[V 1 U],f|[V \eta_\l^{-1} U]).$$ We regard the kernel of $i^\dagger$ as a space of $\l-$new forms. The maps $i$ and $i^\dagger$ commute with Hecke operators $T_v, S_v$, where $v \nmid \n p\l\delta$.

The same double coset operators give maps $$j:L^* \times L^* \rightarrow M^*,$$ $$j^\dagger: M^* \rightarrow L^* \times L^*.$$

Using Proposition \ref{prop:pairinghecke} we have 
$$\langle i(f,g),\lambda\rangle = \langle (f,g), j^\dagger \lambda \rangle$$
for $f,g\in L$, $\lambda \in M^*$. Similarly
$$\langle f,j(\lambda,\mu)\rangle = \langle i^\dagger f, (\lambda,\mu) \rangle$$ for $d \in M$, $\lambda,\mu \in L^*$.

A calculation shows that $i^\dagger i$ acts on the product $L\times L=L^2$ by the matrix (acting on the right) $$\begin{pmatrix}\mathbf{N}\l+1 & [U\varpi_\l^{-1}U][U\eta_\l U] \\ [U\eta_l U] & \mathbf{N}\l+1 \end{pmatrix}=\begin{pmatrix}\N\l+1 & S_\l^{-1}T_\l \\ T_\l & \N\l+1 \end{pmatrix}.$$ We have exactly the same double coset operator formula for the action of $j^\dagger j$ on the product $L^*\times L^* = L^{*2}$. The Hecke operators $S_\l, T_\l$ act by $[U\varpi_\l^{-1}U], [U\eta_\l^{-1} U]$ respectively on $L^*$. Also, the double coset $U\eta_\l U$ is the same as $U\varpi_\l \eta_\l^{-1}U$, since the matrix which is the identity at every factor except $\l$ and $\begin{pmatrix}
0 & 1\\1&0
\end{pmatrix}$ at $\l$ is in U. From these two facts we deduce that, in terms of Hecke operators, $j^\dagger j$ acts on $L^*\times L^*$ by the matrix (again acting on the right) $$\begin{pmatrix}\N\l+1 & T_\l \\ S_\l^{-1}T_\l & \N\l+1 \end{pmatrix}.$$

\begin{lemma}\label{chentf}
Suppose $A$ is an integral domain, which we assume to be normal and equidimensional of dimension $d$. Suppose $B$ is an $A$-algebra which is integral over $A$ and torsion free. Then $B$ is equidimensional of dimension $d$.
\end{lemma}
\begin{proof}
Recall that a Noetherian ring $R$ (of finite Krull dimension $d$) is equidimensional of dimension $d$ if $R/\p$ has Krull dimension $d$ for all minimal prime ideals $\p$ of $R$. The lemma follows from applying the going up and going down theorems. Note that we do not need to assume that $B$ is an integral domain --- the $A$-torsion free condition suffices for going down (see, for example, \cite[Ch. V, Theorem 6]{ZS}).
\end{proof}
\begin{proposition}\label{propinj} Suppose that $X$ is an admissible affinoid open in $\mathscr{W}$. Then the map $i^\dagger i$ is injective. \end{proposition}
%
\begin{proof}
We assume without loss of generality that $X$ is connected. Let $L_0$ be the $\OO(X)$-module $\ker(i^\dagger i)$. The submodule $L_0\subset L^2$ is stable under the action of $\mathbb{T}^{(\n p\delta)}[U_\pi]$, as is the image, denoted $L_1$, of $L_0$ in $L$ under the first projection map from $L^2$. We denote by $\mathscr{H}_L$ the image of $\mathbb{T}^{(\n p\delta)}[U_\pi]$ in $\End_{\OO(X)}(L)$, and denote by $\mathscr{H}_{L_1}$ the image of $\mathbb{T}^{(\n p\delta)}[U_\pi]$ in $\End_{\OO(X)}(L_1)$ (which is naturally a quotient of $\mathscr{H}_L$). 

Suppose that $L_0$ is non-zero. For $(f,g)$ in $L_0$ we have $(\N\l+1)f+S_\l^{-1}T_\l g=T_\l f+(\N\l+1)g=0$. Eliminating $g$ we get $T_\l^2 f-(\N\l+1)^2 S_\l f=0$, so the Hecke operator $T_\l^2-(\N\l+1)^2 S_\l$ maps to $0$ in $\mathscr{H}_{L_1}$. Since $L_1$ is $\OO(X)$-torsionfree, $\mathscr{H}_{L_1}$ (and $(\mathscr{H}_{L_1})^{red}$) is $\OO(X)$-torsionfree. Applying lemma \ref{chentf} we conclude that $(\mathscr{H}_{L_1})^{red}$ is equidimensional of dimension $\dim(\OO(X))$. We now have a closed immersion of reduced rigid spaces $Sp(\mathscr{H}_{L_1}^{red}) \rightarrow Sp(\mathscr{H}_{L}^{red})$, where the source and target are equidimensional of the same dimension. It follows from \cite[Proposition 1.2.3]{CM} (see also \cite[Corollary 2.2.7]{conirr}) that the image of this closed immersion is a union of irreducible components of $Sp(\mathscr{H}_{L}^{red})$. Therefore the zero locus of the Hecke operator $T_\l^2-(\N\l+1)^2 S_\l$ contains an irreducible component of the eigenvariety $\mathscr{E}^D(\n)$, so by Theorem \ref{Zardens} this Hecke operator vanishes on an essentially classical point. Recalling that the definition of essentially classical involves a twist by a character which is  unramified at $\l$, it follows that $T_\l^2-(\N\l+1)^2 S_\l$ vanishes at a classical point, which violates the Hecke eigenvalue bounds given by the Ramanujan-Petersson conjecture for Hilbert modular forms (\cite[Theorem 1]{Blasius}) (alternatively, the local factor at $\l$ of a cuspidal automorphic representation of $\GL_2(\A_F)$ must be generic, and the Hecke eigenvalue condition implies that we are in the non-generic principal series). We obtain a contradiction, so the kernel $L_0$ must be zero.\end{proof}

Note that the injectivity of $i^\dagger i$ implies the injectivity of $i$. The above shows that if $X$ is as in the statement of Proposition \ref{propinj}, we have $\ker(i^\dagger)\cap \mathrm{im}(i)=0$ so the forms in $M$ cannot be both old and new at $\l$. However, if we look at overconvergent automorphic forms with weight a smaller dimensional affinoid in $\mathscr{W}$ (which does not contain a Zariski dense set of classical weights), then $i^\dagger i$ may have a kernel - this corresponds to families of $p$-adic automorphic forms which are both old and new at $\l$.

\subsection{Some modules}\label{somemodules}
In this section we assume that $X$ is an admissible open affinoid in $\mathscr{W}$. We denote the fraction field of $\OO(X)$ by $E$. If $A$ is an $\OO(X)$-module we write $A_E$ for the $E$-vector space $A\otimes_{\OO(X)}E$.

We begin this section by noting that the injectivity of $i^\dagger i$ implies the injectivity of $j^\dagger j$:

Suppose $j^\dagger j (\lambda,\mu)=0$. Then $\langle (f,g),j^\dagger j (\lambda,\mu)\rangle=0$ for all $(f,g)\in L_E$, so (by Proposition \ref{prop:pairinghecke}) $\langle i^\dagger i(f,g),(\lambda,\mu)\rangle =0$ for all $(f,g)\in L_E$. Now since $i^\dagger i:L_E\rightarrow L_E$ is an injective endomorphism of a finite dimensional vector space, it is an isomorphism, so we see that $\lambda=\mu=0$. Hence $j^\dagger j$ (thus a fortiori $j$) is injective.

We now define two chains of modules which will prove useful:
$$\begin{array}{l l}\Lambda_0:=L^2 & \Lambda_0^*:=L^{*2}\\
\Lambda_1:=i^\dagger M & \Lambda_1^*:=j^\dagger M^*\\
\Lambda_2:=i^\dagger(M\cap i(L_E^2)) & \Lambda_2^*:=j^\dagger(M^*\cap j((L^*_E)^2))\\
\Lambda_3:=i^\dagger i L^2 & \Lambda_3^*:=j^\dagger j L^{*2}.\end{array}$$
We note that $\Lambda_0\supset \Lambda_1 \supset \Lambda_2 \supset \Lambda_3$, and that $$\Lambda_2/\Lambda_3=i^\dagger(M\cap i(L_E^2)/iL^2)=i^\dagger((M/iL^2)^{\mathrm{tors}}),$$ with analogous statements for the starred modules.

We fix the usual action of $\mathbb{T}^{(\n\delta p \l)}$ on all these modules. We can now describe some pairings between them which will be equivariant under the $\mathbb{T}^{(\n\delta p \l)}$ action. They will not all be equivariant with respect to the action of $T_\l$.

We have a (perfect) pairing $\langle,\rangle:L_E^2\times (L^*_E)^2\rightarrow E$ which, since $j$ is injective, induces a pairing $$\Lambda_0 \times (M^*\cap j((L^*_E)^2)) \rightarrow E/\OO(X),$$ which in turn induces a pairing $$P_1: \Lambda_0/\Lambda_1 \times (M^*\cap j((L^*_E)^2)/j(L^{*2})) \rightarrow E/\OO(X).$$ The fact that this pairing is perfect follows from \cite[Lemma 6]{chicomp} (the proof of which applies verbatim to the more general setting of this paper).

In exactly the same way, we have a perfect pairing 
$$P_2: (M\cap i(L_E^2))/i(L^2) \times \Lambda_0^*/\Lambda_1^* \rightarrow E/\OO(X).$$

The final pairing we will need is induced by the pairing between $M$ and $M^*$. It is straightforward to check that this gives a perfect pairing: $$P_3: \mathrm{ker}(i^\dagger) \times M^*/(M^* \cap j((L_E^*)^2)) \rightarrow \OO(X).$$ 
\subsection{An analogue of Ihara's lemma}\label{Ihara}

We can generalise the version of Ihara's lemma appearing in \cite{chicomp} to this setting. We want to obtain information about the $\mathbb{T}^{(\n\delta p \l)}$ action on the quotients $\Lambda_2/\Lambda_3\cong i^\dagger(M/iL^2)^{\mathrm{tors}}$, $\Lambda_2^*/\Lambda_3^*\cong j^\dagger(M^*/jL^{*2})^{\mathrm{tors}}$, $\Lambda_0/\Lambda_1$ and $\Lambda_0^*/\Lambda_1^*$. 

\begin{lemma}\label{normtrivial} Let $Y \hookrightarrow X$ be a closed immersion with $Y$ a reduced irreducible affinoid. Let $y\in \OCY$ be non-zero. Suppose $y$ factors through $Nm$, that is $y(g)=y(h)$ for all $g, h \in D_f^\times$ with $Nm(g)=Nm(h)$. Then the map \[\nn:\OO_p^\times\rightarrow\OO(Y)^\times\] associated with $Y$ is a finite order character and there exists a finite \'{e}tale cover $Y'\rightarrow Y$, a finite Abelian extension $F'/F$, a compact open subgroup $U^p \subset \A_{F,f}^{p,\times}$ and a continuous character \[\psi: F^\times\backslash \A_{F}^\times/\overline{U^pF_\infty^{\times,+}} \rightarrow \OO(Y')^\times\] such that if $v \nmid \n p\delta \l$ is a prime of $F$ which splits completely in $F'$ then $\psi(\varpi_v)$ is in $\OO(Y)^\times$ and $T_v y =\psi(\varpi_v)(\mathbf{N}v+1)y$. The field $F'$ (and the cover $Y'$) may be chosen independent of $y$.

\end{lemma}
\begin{proof}
Suppose $y$ is as in the statement of the lemma. First we show that $\nn$ is a finite order character. Suppose we have $u \in U$ such that $u_p \in \SL_2(F_p)$ and $u_v$ is the identity away from $p$.  Then, since $Nm(u)=1$, $y(g)=y(gu)=y(g)\cdot u_p$ for all $g \in D^\times_f$. If we denote by $u_a$ the element of $U$ with $p$ factor equal to $\begin{pmatrix}
1 & a\\ 0 & 1
\end{pmatrix}$ for $a \in \OO_p$, and other factors equal to the identity, then we see that $y(g)(z,x)=y(g)\cdot u_{a,p}(z,x)=y(g)(z+a,x)$ for all $a \in \OO_p$, $z \in \mathbb{B}_{r}$, $x \in Y$, so $y(g)(z,x)$ is constant in $z$ (for fixed $x$), since non-constant rigid analytic functions have discrete zero sets. 

Recall that $U=\UU$, so if we denote be $u_0$ the element of $U$ with $p$ factor equal to $\begin{pmatrix}
1 & 0\\ \pi^\alpha & 1
\end{pmatrix}$ and other factors the identity, then again we have $y(g)=y(g)\cdot u_{0,p}$. This implies that $y(g)\cdot u_{0,p}(z,x)$ is, for fixed $x\in Y$, a constant function of $z$. Let $n$ be a positive integer such that $|\pi_j^{n+\alpha_j}|\le r_j|\pi_j^{\alpha_j}|$ for all $j \in J$. We have $$(y(g)\cdot u_{0,p})(z,x) = \nn(\pi^\alpha z + 1,x)y(g)(z,x),$$ which implies that the character $\nn:\OO_p^\times \rightarrow \kappa(x)^\times$ is trivial on the subgroup $1+\pi_j^{\alpha_j + n} \OO_j$ of each factor $\OO_j^\times$ (we first deduce this for the characters $\nn_x$ with $x$ such that $y(g)(z,x)$ is non-zero, then extend to all $x\in Y$ by analytic continuation, since $Y$ is irreducible).  Hence $\nn$ is a finite order character. This implies that $\vv$ is trivial on a finite index subgroup, $\Gamma$, of $\OO_F^\times$.

Also, for each $g \in D^\times_f$ we have $y(g)\cdot \gamma = \vv(\det \gamma)y(g)$ for all $\gamma$ in the projection of $U_1(\n\pi^{\alpha+n})$ to $\mathbb{M}_{\alpha + n}$, since these matrices all have bottom right hand entry congruent to $1 \mod \pi^{\alpha+n}$. 

Since $\mathbb{G}_m/F$ satisfies the congruence subgroup property (\cite[Th\'{e}or\`{e}me 1]{Chevalley}), we have $\OO_F^\times\cap U^p(1+\pi^r\OO_p)F_\infty^{\times,+} \subset\Gamma$ for some compact open subgroup $U^p \subset \A_{F,f}^{p,\times}$ and $r \in \Z_{\ge 1}$. Therefore, setting $\Gamma' = \OO_F^\times\cap U^p(1+\pi^r\OO_p)F_\infty^{\times,+}$, we have an injective map \[\overline{\Gamma}'\backslash(1+\pi^r\OO_p) \hookrightarrow F^\times\backslash\A_F^\times/\overline{U^pF_\infty^{\times,+}}\] with finite cokernel. Therefore we have $Y'$ and $\psi$ as in the statement of the lemma, with $\psi|_{(1+\pi^r\OO_p)}=\vv|_{(1+\pi^r\OO_p)}^{-1}$.

We let $F'$ be the class field for the (narrow) ray class group 
\[F^\times\backslash\A_F^\times/U^p(1+\pi^r\OO_p)F_\infty^{\times,+},\]
so if $v$ splits completely in $F'$ we do indeed have $\psi(\varpi_v) = \vv(x) \in \OO(Y)^\times$ for $x \in 1+\pi^r\OO_p$ with $F^\times \varpi_v U^pF_\infty^{\times,+} = F^\times x U^pF_\infty^{\times,+}$ (such an $x$ exists because $\varpi_v$ has trivial image in the narrow ray class group).


Suppose a finite place $v$ of $F$ splits completely in $F'$ and is coprime to $\n\delta p\l$. Since $Nm(U_1(\n\pi^{\alpha+n})) =\OO_D\otimes_\Z\widehat{\Z}$, we may choose $u_0$ in $U_1(\n\pi^{\alpha+n})$ such that we have an equality in $\A_F^\times$ \[\varpi_v = a Nm(u_0)\alpha\] for some $a \in F^\times$, $\alpha \in F_\infty^{\times,+}$. Therefore we have an equality $\varpi_v = aNm(u_0)$ in $\A_{F,f}^\times$, with $a \in F^{\times,+}$. The image of the reduced norm map on $D^\times$ is precisely $F^{\times,+}$, by the Hasse--Schilling--Maass theorem \cite[33.15]{Reiner}, so we have $d \in D^\times$ such that $Nm(\eta_v)=\varpi_v = Nm(d u_0)$.

Now we can compute
\begin{align*}T_v(y)(g)&=\sum_{u\in(\eta_v^{-1} U\eta_v)\cap U\backslash U}y(gu^{-1} \eta_v^{-1})\cdot u_p\\
&=\sum_{u\in(\eta_v^{-1} U\eta_v)\cap U\backslash U}y(g\eta_v^{-1}u^{-1})\cdot u_p\\
&=\sum_{u\in(\eta_v^{-1} U\eta_v)\cap U\backslash U}y(g\eta_v^{-1})=(\mathbf{N}v+1)y(g\eta_v^{-1})\\
&=(\mathbf{N}v+1)y(gu_0^{-1})=\vv(\det u_{0,p}^{-1})(\mathbf{N}v+1)y(g).\end{align*}

By the definition of $\psi$, we have \[\vv(\det u_{0,p}^{-1}) = \psi(\det u_{0,p}) = \psi(Nm(d u_0)) = \psi(\varpi_v),\] so we are done.
\end{proof}
\begin{lemma}\label{normtriviald} Let $Y \hookrightarrow X$ be a closed immersion with $Y$ a reduced irreducible affinoid. Let $y\in \OCdY$. If $y$ factors through $Nm$, then $y=0$.
\end{lemma}
\begin{proof}
This is similar to the proof of \cite[Lemma 7 (ii)]{chicomp}, but we rely on an explicit computation with $\mathscr{A}_{Y,r}$ which becomes more complicated when the dimension of $\mathbb{B}_r$ (which is equal to $[F:\Q]$) grows.

Without loss of generality (by extending $K$ if necessary) we can assume that for every $j \in J$, $r_j = |x_j|$ for some $x_j \in K$. With this assumption, we define an isomorphism $$\tilde{\nu}_r^*:\mathscr{A}_{Y,r}\cong\prod_{s \in S} \OO(Y)\langle T_{s,1},...,T_{s,g} \rangle$$ as in Section \ref{explicit}, where $S \subset \OO_p$ is a finite set defined as in that section. Write monomials in $\OO(Y)\langle T_{s,1},...,T_{s,g} \rangle$ as $T_s^{\mathbf{j}}$, where $\mathbf{j}=(j_1,...,j_g)\in (\Z_{\ge 0})^I$ and $T_s^{\mathbf{j}}=\prod_{i\in I}T_{s,i}^{j_i}.$

We denote the inverse of $\tilde{\nu}_r^*$ by $\theta$. By duality, this induces an isomorphism $\theta^*$ between $\mathscr{D}_{Y,r}$ and $\prod_{s \in S} \OO(Y)\langle[\tau_{s,1},...,\tau_{s,g}]\rangle$, where $\OO(Y)\langle[\tau_{s,1},...,\tau_{s,g}]\rangle$ denotes a ring of power series with bounded coefficients in $\OO(Y)$ (as always, we use the supremum norm on the reduced affinoid algebra $\OO(Y)$). For monomials in $\OO(Y)\langle[\tau_{s,1},...,\tau_{s,g}]\rangle$ we write $\tau_s^{\mathbf{l}}=\prod_{k\in I}\tau_{s,k}^{l_i}$ for $\mathbf{l}\in (\Z_{\ge 0})^I$ and the pairing between $\mathscr{A}_{Y,r}$ and $\mathscr{D}_{Y,r}$ is induced component-wise by $\langle  T_s^{\mathbf{j}}, \tau_s^{\mathbf{l}}\rangle = \delta_{\mathbf{j},\mathbf{l}}$. Here $\delta_{\mathbf{j},\mathbf{l}}$ denotes the function which is $1$ if $\mathbf{j}=\mathbf{l}$ and $0$ otherwise. 

Now we can compute the action of $\gamma=\begin{pmatrix}
1 & -\pi^N\\0 & 1
\end{pmatrix} \in \mathbb{M}_\alpha$ on $\mathscr{D}_{Y,r}$, where $N$ is an integer chosen to be sufficiently large that $|\pi_j|^N < r_j$ for all $j$. We use lemma \ref{simplecalc}. Let $f$ be an element of $\mathscr{D}_{Y,r}$, with $\theta^*(f) = (F_s)_{s \in S}$ where $F_s = \sum b^\mathbf{l}_{s} \tau_s^\mathbf{l}.$ We have \begin{align*}\langle \theta(T^\mathbf{j}_{s}),f\cdot\begin{pmatrix}
1 &  -\pi^N\\0 & 1
\end{pmatrix}\rangle &= \langle \theta(T^\mathbf{j}_{s})\cdot\begin{pmatrix}
1 & \pi^N\\0 & 1\end{pmatrix}, f\rangle =\langle \tilde{\nu}_r^*(\theta(T^\mathbf{j}_{s}))\cdot\begin{pmatrix}
1 & \pi^N\\0 & 1\end{pmatrix}), (F_s)_{s \in S} \rangle.\end{align*}
Now by lemma \ref{simplecalc} we have $$\tilde{\nu}_r^*(\theta(T^\mathbf{j}_{s}))\cdot\begin{pmatrix}
1 & \pi^N\\0 & 1\end{pmatrix}) = \prod_{i \in I}(T_{s,i}+\pi_i^N/x_i)^{j_i}.$$ 
\begin{sublemma}
Suppose $$f\cdot\begin{pmatrix}
1 &  -\pi^N\\0 & 1
\end{pmatrix}=f.$$ Then $f=0$.
\end{sublemma}
\begin{proof}
We show that for every monomial $T^\mathbf{j}_{s}$ we have $$b^{\mathbf{j}}_s=\langle \theta(T^\mathbf{j}_{s}),f\rangle= \langle \theta(T^\mathbf{j}_{s}),f\cdot\begin{pmatrix}
1 &  -\pi^N\\0 & 1
\end{pmatrix}\rangle=0.$$ 
We induct on $(\Z_{\ge 0})^I$ using colexicographic ordering (which is a well-ordering, with minimal element $(0,...,0)$). Note that we fix an ordering of $I$ (e.g. as we have done previously we label the elements of $I$ as $1,...,g$). We recall the definition of the colexicographic ordering.
\begin{definition}
Given $\mathbf{j}=(j_1,...,j_g),\mathbf{j}'=(j'_1,...,j'_g) \in (\Z_{\ge 0})^I$ we say that $$\mathbf{j}'<^{\mathrm{colex}}\mathbf{j}$$ if there is an $m\in I$ such that for all $i > m$ we have $j_i=j'_i$ and we have $j_m>j'_m$. 
\end{definition}
We now proceed with our inductive proof. 

Suppose we have $\mathbf{j}=(j_1,...,j_g)\in (\Z_{\ge 0)})^I$, and for all $\mathbf{j}' \in (\Z_{\ge 0)})^I$ with $\mathbf{j}'<^{\mathrm{colex}}\mathbf{j}$ we have $b^{\mathbf{j}'}_s=0$. Set $\mathbf{j}^+=(j_1+1,j_2,...,j_g)$. Then we have \begin{align*} b^{\mathbf{j}^+}_s &= \langle\theta(T^{\mathbf{j}^+}_s),f\rangle \\&= \langle\theta(T^{\mathbf{j}^+}_s),f\cdot\begin{pmatrix}
1 &  -\pi^N\\0 & 1
\end{pmatrix}\rangle = \langle (T_{y,1}+\pi_1^N/x_1)^{j_1+1}\prod_{i =2}^{g}(T_{y,i}+\pi_i^N/x_i)^{j_i}, F_s \rangle.\end{align*}
Now this final pairing is just equal to 
$$\langle T_s^{\mathbf{j}^+} + (j_1+1)(\pi_1^N/x_1)T_s^{\mathbf{j}},F_s\rangle$$ since all other monomials in the left hand side of the pairing have index $<^{\mathrm{colex}} \mathbf{j}$, hence by the inductive hypothesis pair to $0$ with $F_s$. So we obtain $b^{\mathbf{j}^+}_s = b^{\mathbf{j}^+}_s + (j_1+1)(\pi_1^N/x_1)b^{\mathbf{j}}_s$, so $b^{\mathbf{j}}_s = 0$ as required. Note that this simultaneously handles the `base case' $\mathbf{j}=(0,...,0)$ and the `inductive step'.
\end{proof}

Now we return to the statement in the lemma and suppose $y\in \OCdY$ factors through $Nm$. Let $u_1$ be the element of $U \subset D_f^\times$ with $p$ component equal to $\gamma = \begin{pmatrix}
1 &  -\pi^N\\0 & 1
\end{pmatrix}$ and all other components the identity. For all $g \in D_f^\times$, $y(gu_1)=y(g)\gamma$ since $u_1 \in U$ and $y(gu_1)=y(g)$ since $Nm(u_1)=1$. Hence $y(g)=y(g)\gamma$ and the above calculation shows that $y(g)=0$ for all $g$.
\end{proof}
\begin{remark}
Note that there is a rather odd asymmetry in the preceding lemma. We attempt to give some explanation of this. Denote by $\mathscr{A}_{Y,r}^\mathrm{triv} \subset \mathscr{A}_{Y,r}$ the subspace of functions which are constant in the variable $z \in \mathbb{B}_r$. This subspace is preserved under the action of $\mathbb{M}_\alpha$, for weights $\kappa$ as specified in part (1) of the above lemma. Dual to the embedding $\mathscr{A}_{Y,r}^\mathrm{triv} \rightarrow \mathscr{A}_{Y,r}$ there is a $\mathbb{M}_\alpha^{-1}$-equivariant projection $\mathscr{D}_{Y,r} \rightarrow \mathscr{D}_{Y,r}^\mathrm{triv}$, where $\mathscr{D}_{Y,r}^\mathrm{triv} = (\mathscr{A}_{Y,r}^\mathrm{triv})^*$. The point of the above calculation is basically to show that this projection does not admit an $\mathbb{M}_\alpha^{-1}$-equivariant section $\mathscr{D}_{Y,r}^\mathrm{triv} \rightarrow \mathscr{D}_{Y,r}$. This fact presumably has a representation--theoretic interpretation.
\end{remark}

\subsection{Very Eisenstein modules}
We will soon apply the preceding lemma to give a form of Ihara's lemma for modules of overconvergent automorphic forms and dual overconvergent forms. First we need a technical definition to take care of the exceptional case in lemma \ref{normtrivial}.

\begin{definition}\label{eismodule}
Denote by $\mathscr{H}$ the image of $\mathbb{T}^{(\delta p\n\l)}$ in $\End(M)$. Suppose $N$ is an $\mathscr{H}$-module which is finitely generated as an $\OO(X)$-module. We say that $N$ is \emph{very Eisenstein} if every prime ideal $\mathfrak{P}$ of $\cH$ in the support of $N$, with $\p = \mathfrak{P}\cap\OO(X)$, satisfies:
\begin{itemize}
\item  the induced character \[\nn: \OO_p^\times \rightarrow (\OO(X)/\p)^\times\] has finite order
\item there exists a finite \'{e}tale $K$-algebra extension $\OO(X)/\p \hookrightarrow A$, a finite Abelian extension $F'/F$ and a continuous character $\psi: F^\times\backslash \A_{F}^\times/U_1(\n)^pF_\infty^{\times,+} \rightarrow A^\times$ such that if $v \nmid \n p\delta \l$ is a prime of $F$ which splits completely in $F'$ then $\psi(\varpi_v)$ is in $(\OO(X)/\p)^\times$ and $T_v-\psi(\varpi_v)(\mathbf{N}v+1) = 0$ in $\cH/\mathfrak{P}$. 
\end{itemize}  
\end{definition}
\begin{remark}
Since minimal elements of $\mathrm{Supp}_{\cH}(N)$ are the same as minimal elements of $\mathrm{Ass}_{\cH}(N)$ (\cite[Theorem 6.5]{Matsumura}), the very Eisenstein condition can alternatively be checked on associated primes.
\end{remark}
\begin{lemma}\label{sesok}
Suppose we have a sequence of $\cH$-modules \[A \rightarrow B \rightarrow C\] which is exact in the middle, and that $A$ and $C$ are very Eisenstein. Then $B$ is very Eisenstein.

Suppose $N$ is a very Eisenstein $\cH$-module. Then $\Hom_{\OO(X)}(N,E/\OO(X))$ is also very Eisenstein.
\end{lemma}
\begin{proof}
Since the support of $B$ is contained in the union of the support of $A$ and the support of $C$, the first part is obvious.

For the second part, since \[\mathrm{ann}_{\cH}(N) \subset \mathrm{ann}_{\cH}(\Hom_{\OO(X)}(N,E/\OO(X)))\] the support of $\Hom_{\OO(X)}(N,E/\OO(X))$ is contained in the support of $N$, so again the result is obvious.
\end{proof}

\begin{lemma}\label{preihara}
Let $Y \hookrightarrow X$ be a closed, reduced and irreducible sub-affinoid. Then the module $\mathrm{Tor}_1^{\OO(X)}(M/iL^2,\OO(Y))$ is very Eisenstein.
\end{lemma}
\begin{proof}

We have a short exact sequence
$$\xymatrix{
0\ar[r]&L^2\ar[r]^i&M\ar[r]&M/iL^2\ar[r]&0
}.$$
Noting that $L^2$ and $M$ are projective $\OO(X)$-modules, hence flat, and taking derived functors of $-\otimes_{\OO(X)}\OO(Y)$ gives an exact sequence
$$\xymatrix{
0\ar[r]&\mathrm{Tor}^{\OO(X)}_1(M/iL^2,\OO(Y))\ar[r]^\delta&L^2\otimes_{\OO(X)}\OO(Y)\ar@/^2pc/[d]^i\\
0&M/iL^2\otimes_{\OO(X)}\OO(Y)\ar[l]&M\otimes_{\OO(X)}\OO(Y)\ar[l]
}.$$

To prove the proposition, it suffices to prove that kernel of the map \[i:L^2\otimes_{\OO(X)}\OO(Y)\rightarrow M\otimes_{\OO(X)}\OO(Y)\] is very Eisenstein. We can identify $L\otimes_{\OO(X)}\OO(Y)$ with $\OCY^{Q_Y}$ and $M$ with $\OCVY^{Q_Y}$, so it in fact suffices to prove that the kernel of the natural `level raising' map $i_Y: \OCY^{\oplus 2} \rightarrow \OCVY$ is very Eisenstein.

Suppose $i_Y(y_1,y_2)=0$. Then $y_1=-y_2|\eta_\l$, so we have $y_2 \in \OCY$, $y_2|\eta_\l \in \OCY$. Therefore $y_2$ and $y_2|\eta_\l$ are both invariant under the action of the group $U$, so $y_2$ is invariant under the action of the group generated by $U$ and $\eta_\l U \eta_\l^{-1}$ in $D_f^\times$. It follows as in the proof of \cite[Proposition 8]{chicomp} that $y_1$ and $y_2$ factor through $Nm$ and now lemma \ref{normtrivial} shows that $\mathrm{Tor}_1^{\OO(X)}(M/iL^2,\OO(Y))$ is very Eisenstein.
\end{proof}
\begin{lemma}\label{preiharadual}Let $Y \hookrightarrow X$ be a closed, reduced and irreducible sub-affinoid. The module $\mathrm{Tor}_1^{\OO(X)}(M^*/jL^{*2},\OO(Y))$ is $0$ for all prime ideals $\p$ of $\OO(X)$
\end{lemma}
\begin{proof}
We proceed as in the proof of lemma \ref{preihara}. With the help of lemma \ref{bcimm} and lemma \ref{normtriviald}, we obtain the desired result.
\end{proof}

The following consequence of the preceding two lemmas will be the most convenient analogue of Ihara's lemma for our applications.

\begin{lemma}\label{lemma:trivialtorsion}
\begin{enumerate} \item The $\cH$-module $(M/iL^2)^{\mathrm{tors}}$ is very Eisenstein. 
\item The module $(M^*/jL^{*2})^{\mathrm{tors}}$ is equal to $0$.\end{enumerate}
\end{lemma}
\begin{proof}Since $X$ is a disjoint union of reduced and irreducible admissible open affinoids, we may assume without loss of generality that $X$ is reduced and irreducible (hence $\OO(X)$ is an integral domain). We begin by proving the first part of the lemma. Suppose $\mathfrak{P}$ is an associated prime ideal (in $\cH$) of $(M/iL^2)^{\mathrm{tors}}$. So we have $m \in M/iL^2$ and $0\ne\alpha \in \OO(X)$ with $\alpha m = 0$ and $\mathrm{ann}_{\cH}(m) = \mathfrak{P}$. Since $\alpha$ is not a zero-divisor, the $\cH$-submodule of $\alpha$-torsion elements in $M/iL^2$ is isomorphic to $\mathrm{Tor}^{\OO(X)}_1(M/iL^2,\OO(X)/(\alpha))$. So it suffices to prove that the $\cH$-module $\mathrm{Tor}^{\OO(X)}_1(M/iL^2,\OO(X)/(\alpha))$ is very Eisenstein.

By \cite[Theorem 6.4]{Matsumura} there is a chain $0=M_0\subset M_1\subset\cdots\subset M_n = \OO(X)/(\alpha)$ of $\OO(X)$-submodules such that for each $i$ we have $M_i/M_{i-1} \cong \OO(X)/\p$ with $\p$ a prime ideal of $\OO(X)$. By applying lemma \ref{sesok} and a simple d\'{e}vissage, we are done if we can prove that  $\mathrm{Tor}^{\OO(X)}_1(M/iL^2,\OO(X)/\p)$ is very Eisenstein for any prime ideal $\p$ of $\OO(X)$, which is precisely lemma \ref{preihara}.

The second part of the lemma follows similarly from lemma \ref{preiharadual}.
\end{proof}
\begin{corollary}\label{cor:trivtors}
The modules $\Lambda_2/\Lambda_3$ and $\Lambda_0^*/\Lambda_1^*$ are very Eisenstein. The modules $\Lambda_2^*/\Lambda_3^*$ and $\Lambda_0/\Lambda_1$ are equal to $0$.
\end{corollary}
\begin{proof}
The statement about $\Lambda_2/\Lambda_3$ follows from the first part of lemma \ref{lemma:trivialtorsion}, since $\Lambda_2/\Lambda_3$ is a homomorphic image of $(M/iL^2)^{\mathrm{tors}}$. The statement about $\Lambda_0^*/\Lambda_1^*$ follows from the duality (via pairing $P_2$) between this module and $\Lambda_2/\Lambda_3$. The other two statements similarly follow from the second part of lemma \ref{lemma:trivialtorsion}.
\end{proof}

\section{Raising the level}
\subsection{Very Eisenstein ideals}
\begin{definition}\label{eisideal}
We say that a prime ideal $\mathfrak{P}$ of $\cH$ is \emph{very Eisenstein} if the $\cH$-module $\cH/\mathfrak{P}$ is very Eisenstein. Equivalently, we require $\mathfrak{P}$ to satisfy the itemised conditions in Definition \ref{eismodule}.
\end{definition} 
Recall that if $\mathfrak{M}$ is a maximal ideal of $\cH$, then there is an attached semi-simple Galois representation $\rho_\mathfrak{M}:\Gal(\overline{F}/F) \rightarrow \GL_2(\overline{\Q}_p)$. We normalise things so that the $T_v$ eigenvalue arising from $\mathfrak{M}$ is equal to the trace of $\rho_\mathfrak{M}(\sigma_v)$, where $\sigma_v$ denotes an \emph{arithmetic} Frobenius element. If $\mathfrak{M}$ is very Eisenstein, then this Galois representation is reducible, but the requirement that the character $\nn: \OO_p^\times \rightarrow (\OO(X)/\mathfrak{M})^\times$ is finite order imposes an additional restriction on the Galois representation (that can be translated into a condition on its generalised Hodge--Tate weights).

Recall that $\cH$ denotes the image of $\mathbb{T}^{(\n\delta p\l)}$ in $\End_{\OO(X)}(M)$. We similarly let $\cH_L$ denote the image of $\mathbb{T}^{(\n\delta p)}$ in $\End_{\OO(X)}(L)$.
There is a map $\cH\rightarrow \cH_L$ coming from the embedding $L \hookrightarrow M$ (given by regarding a form of level $U$ as a form of level $V$). If $I$ is an ideal of $\cH_L$ we denote by $I_M$ the inverse image of $I$ in $\cH$. Note that if $N$ is any finitely generated $\cH_L$-module, then $\mathrm{ann}_{\cH}(N) = (\mathrm{ann}_{\cH_L}(N))_M$, and so if $\PP$ is in $\mathrm{Supp}_{\cH_L}(N)$ then $\PP_\N$ is in $\mathrm{Supp}_{\cH}(N)$.

\begin{proposition}\label{supporttheorem}Suppose $\mathfrak{P}$ is a prime ideal of $\cH_L$ such that $\mathfrak{P}_M$ is not very Eisenstein. Moreover, suppose that $\mathfrak{P}$ contains $T_\l^2-(\mathbf{N}\l+1)^2 S_\l$. Then $\mathfrak{P}_M$ is in the support of the $\cH$-module $\mathrm{ker}(i^\dagger)\subset M$.
\end{proposition}
\begin{proof}
Consider the $\cH_L$-module $$Q:=\Lambda_0^*/\Lambda_3^*=L^{*2}\left/
L^{*2}\begin{pmatrix}\mathbf{N}\l+1 & T_\l \\ S_\l^{-1}T_\l & \mathbf{N}\l+1 \end{pmatrix},\right.$$ The $\cH_L$-supports of $L$ and $L^*$ are equal, and $\mathrm{det}\begin{pmatrix}\mathbf{N}\l+1 & T_\l \\ S_\l^{-1}T_\l & \mathbf{N}\l+1 \end{pmatrix}\in \mathfrak{P}$, so $\mathfrak{P}$ is in support of $Q$ (the $\cH_{L,\mathfrak{P}}$-module $Q_\mathfrak{P}/\mathfrak{P}Q_\mathfrak{P}$ is non-zero since it is the cokernel of a map of vector spaces with determinant zero). Thereform $\PP_M$ is in the support of $Q$.

Corollary \ref{cor:trivtors} implies that if $\PP_{M}$ is in the support of $\Lambda_0^*/\Lambda_1^*$ or $\Lambda_2^*/\Lambda_3^*$ then it is very Eisenstein, so it must be in the support of $\Lambda_1^*/\Lambda_2^*$. This quotient is a homomorphic image of $M^*/(M^*\cap j(L_F^{*2}))$, so $\PP_{M}$ is in the support of $M^*/(M^*\cap j(L_F^{*2}))$. Finally we can apply pairing $P_3$ (which is equivariant with respect to the action of $\cH$) to conclude that $\PP_M$ is in the support of $\mathrm{ker}(i^\dagger)$. 
\end{proof}

\section{Applications}\label{apps}
In this section we explain an application of the preceding results.

\subsection{Level raising for $p$-adic Hilbert modular forms}

Let $\mathscr{E}^D(\n)$ be the reduced eigenvariety of tame level $U_1(\n)$ as defined in Definition \ref{eigdef}. Similarly we denote by $\mathscr{E}^{D}(\n\l)$ the reduced eigenvariety of tame level $U_1(\n)\cap U_0(\l)$, where we construct this eigenvariety using the Hecke operators at $\l$ in addition to the usual Hecke operators away from the level. This allows us to relate $\mathscr{E}^{D}(\n\l)$ and the reduction $\mathscr{E}^{D,\l\mhyphen\mathrm{old}}(\n\l)$ of the two-covering of $\mathscr{E}^D$ corresponding to taking roots of the $\l$-Hecke polynomial.

\begin{lemma}\label{OLD}
There is a closed embedding $\mathscr{E}^{D,\l\mhyphen\mathrm{old}}(\n\l)\hookrightarrow \mathscr{E}^{D}(\n\l)$, with image the Zariski closure of the essentially classical $\l$-old points in $\mathscr{E}^{D}(\n\l)$.
\end{lemma}
\begin{proof}
This is proved exactly as \cite[Lemma 14]{chicomp}.
\end{proof}

We need one more definition before we can state our main theorem. Recall that $\mathscr{E}^{D}(\n\l)$ has an admissible open cover by affinoids $Sp(\mathbf{T}^{\mathrm{red}})$, where $\mathbf{T}$ is the image of $\mathbb{T}^{(\n p\delta)}[U_\pi]$ in $\End(\OCV^Q)$, with $X \subset \mathscr{W}$ and admissible open affinoid, $Q$ a suitable polynomial factor of the characteristic power series of $U_\pi$ on $\OCV$ and $V = \VV$ for suitable $r,\alpha$. For $\mathbf{T}$ of this form, we define $\mathbf{T}^{\l\mhyphen\mathrm{new}}$ to be the quotient of $\mathbf{T}$ given by restricting the Hecke operators to $\ker(i^\dagger)\subset \OCV^Q$.

\begin{proposition}\label{neweig}
There exists a unique reduced closed subspace \[\iota:\mathscr{E}^{D}(\n\l)^{\l\mhyphen\mathrm{new}} \hookrightarrow \mathscr{E}^{D}(\n\l)\] such that pulling $\iota$ back to the admissible open $Sp(\mathbf{T}^{\mathrm{red}})$ of $\mathscr{E}^{D}(\n\l)$ gives the closed immersion \[Sp(\mathbf{T}^{\l\mhyphen\mathrm{new},\mathrm{red}})\hookrightarrow Sp(\mathbf{T}^{\mathrm{red}}).\]

Moreover, $\mathscr{E}^{D}(\n\l)^{\l\mhyphen\mathrm{new}}$ is equidimensional of dimension $\dim(\mathscr{W})$, and is therefore a union of irreducible components of $\mathscr{E}^{D}(\n\l)$.
\end{proposition}
\begin{proof}
We simply need to glue the closed immersions \[Sp(\mathbf{T}^{\l\mhyphen\mathrm{new}})\hookrightarrow Sp(\mathbf{T})\] and then pass to reduced spaces. 

For $Y \subset X$ an admissible open sub-affinoid, recall that there are Hecke equivariant isomorphisms $\OC^Q\otimes_{\OO(X)}\OO(Y) \cong \OCY^{Q_Y}$ and $\OCV^Q\otimes_{\OO(X)}\OO(Y) \cong \OCVY^{Q_Y}$ which allow us to identify $\ker(i^\dagger)\otimes_{\OO(X)}\OO(Y)$ with \[\ker[i^\dagger_Y: \OCVY^{Q_Y} \rightarrow (\OCY^{Q_Y})^{\oplus 2}].\] Therefore the image of $\mathbb{T}^{(\n p\delta)}[U_\pi]$ in $\End(\ker(i^\dagger_Y))$ can be identified with $\mathbf{T}\otimes_{\OO(X)}\OO(Y)$. This is enough to glue our closed immersions (we leave the details to the reader). Note that since all our $\OO(X)$-modules are finitely generated, the ordinary tensor products $\otimes_{\OO(X)}\OO(Y)$ in the above can be replaced with completed tensor products \cite[3.7.3, Proposition 6]{BGR}.

The fact that $\mathscr{E}^{D}(\n\l)^{\l\mhyphen\mathrm{new}}$ is equidimensional of dimension $\dim(\mathscr{W})$ follows from the fact that (for irreducible $X$) the $\OO(X)$-module $\ker(i^\dagger)$ is $\OO(X)$-torsion free, so we can again apply lemma \ref{chentf}, as in the proof of proposition \ref{propinj}.
\end{proof}
We can now state the main theorem of this paper.
\begin{theorem}\label{raise}
Suppose we have a point $\phi \in \mathscr{E}^D(\n)$ which is not very Eisenstein\footnote{i.e. it is associated with a maximal ideal of a Hecke algebra which is not very Eisenstein}, and with $T_\l^2(\phi)-(\mathbf{N}\l+1)^2S_\l(\phi)=0$. Let the roots of the $\l$-Hecke polynomial corresponding to $\phi$ be $\alpha$ and $\mathbf{N}\l \alpha$ where $\alpha \in \C_p^\times$. Then the point over $\phi$ of $\mathscr{E}^{D,\l-\mathrm{old}(\n\l)}$ corresponding to $\alpha$ also lies in $\mathscr{E}^{D}(\n\l)^{\l\mhyphen\mathrm{new}}$.
\end{theorem}
\begin{proof}
It follows from the construction of the eigenvariety that there exists an admissible affinoid open $X \subset \mathscr{W}$, an $\alpha$, an $r$ and a $Q$ such that the point $\phi$ corresponds to a maximal ideal $\M_\phi$ in the support of $L = \OC^{Q}$, with $U=\UU$.

Now proposition \ref{supporttheorem} applies, so $\mathfrak{M}_M$ is in the support of \[\mathrm{ker}(i^\dagger)\subset M = \OCV^{Q}.\] Therefore $\mathscr{E}^{D}(\n\l)^{\l\mhyphen\mathrm{new}}$ contains a point $\phi'$ with the same Hecke eigenvalues (away from $\l$) as $\phi$. A calculation using the fact that $\phi'$ comes from an eigenform in the kernel of $i^\dagger$ shows that the point $\phi'$ corresponds to the root $\alpha$.
\end{proof}
\begin{remark}
Note that the same conclusion as the above Theorem holds for a \emph{family} of points of $\mathscr{E}^D(\n)$ on which $T_\l^2-(\mathbf{N}\l+1)^2S_\l$ vanishes. Since this Hecke operator is non-vanishing on essentially classical points, it cuts out a codimension one subspace on each irreducible component of $\mathscr{E}^D(\n)$ where it is not invertible.
\end{remark}
\begin{remark}
Recall that the eigenvarieties $\mathscr{E}^D(\n)$ and $\mathscr{E}^D(\n\l)^{\l\mhyphen\mathrm{new}}$ are constructed using Hecke operators away from $\n$. We can replace these eigenvarieties by the `full' reduced eigenvarieties constructed using Hecke operators at all finite places (including the places $v$ dividing $\delta$ --- here the Hecke operator is given by the double coset operator $[U\varpi_v U]$, where $\varpi_v \in D_v^\times \hookrightarrow D_f^\times$ is a uniformiser of $\OO_{D_v}$), and obtain exactly the same statement as Theorem \ref{raise}.
\end{remark}
\subsection{Examples}\label{example}
\subsubsection{Examples if Leopoldt's conjecture does not hold for $(F,p)$}\label{leoeg} Suppose $[F:\Q]$ is even, and let $D/F$ be a definite quaternion algebra non-split at all infinite places and split at all finite places. 

\begin{proposition}
Suppose that Leopoldt's conjecture does not hold for $(F,p)$. Then there is a point $x_L$ of $\mathscr{E}^D(\OO_F)$, lying over the point of $\mathscr{W}$ corresponding to the character \begin{align*}\kappa:\OO_p^\times \times \OO_p^\times &\rightarrow K^\times\\(a,b) &\mapsto \prod_{i\in I} a_i^{-2}b_i\end{align*} such that $T_v(x_L) = 1+(\mathbf{N}v)^{-1}$ and $S_v(x_L) = (\mathbf{N}v)^{-2}$ for all finite places $v$ of $F$ with $v \nmid p$, and $U_{\pi}(x_L)=1$. In particular, the semisimple representation of $G_F = \Gal(\overline{F}/F)$ attached to $x_L$ is $\mathbf{1} \oplus \chi_{cyc}^{-1}$, where $\mathbf{1}$ denotes the trivial character and $\chi_{cyc}$ denotes the $p$-adic cyclotomic character.

For any prime $\l$ of $F$ which is coprime to $p$ the point of $\mathscr{E}^{D,\l-old}(\l)$ lying over $x_L$ which corresponds to the root $(\mathbf{N}\l)^{-1}$ of the Hecke polynomial $X^2-(1+(\mathbf{N}\l)^{-1})X + (\mathbf{N}\l)^{-1}$ also lies in $\mathscr{E}^{D}(\n\l)^{\l\mhyphen\mathrm{new}}$.
\end{proposition}
\begin{proof}
It is explained in \cite[\S 4.3]{Ventullo} that the desired system of Hecke eigenvalues arises from the cuspidal ordinary $\Lambda$-adic Hecke algebra for $\GL_2/F$. There are then several ways to see that this system of eigenvalues must also appear in (an ordinary component of) the eigenvariety $\mathscr{E}^D(\OO_F)$.  

For example, one can apply \cite[Th\'{e}or\`{e}me 1]{ChenJL} to systems of (in this case, finitely generated) Banach modules provided by, on the one hand, the ordinary part of the modules defined in section \ref{ocforms} (with the weights varying in the one dimensional family with $\vv$ trivial and $\nn$ `parallel') and on the other hand, the generic fibre of the ordinary part of the modules of Katz $p$-adic Hilbert modular forms.

Note that the $T_v$ and $S_v$ eigenvalues ($t_v$ and $s_v$ respectively) can be read off from the characteristic polynomial of an arithmetic Frobenius element acting on the Galois representation attached to $x_L$, which is $X^2-t_vX + (\mathbf{N}v)s_v$.

Now for the last part of the proposition we can immediately observe that the point $x_L$ appearing in the above proposition satisfies the conditions of Theorem \ref{raise} for any prime $\l$ of $F$ which is coprime to $p$. 
\end{proof}

Ribet's method can be applied to the family of Galois (pseudo-)representations of $G_F=\Gal(\overline{F}/F)$ attached to an irreducible component $\mathscr{C}$ of $\mathscr{E}^{D}(\l)^{\l\mhyphen\mathrm{new}}$ passing through the point of  $\mathscr{E}^{D}(\l)^{\l\mhyphen\mathrm{new}}$ lying over $x_L$ which is provided by Theorem \ref{raise}. We therefore obtain the following:

\begin{corollary}\label{extn}There exists a non-split extension $V(x_L)$ of $G_F$-representations over a finite extension $K/\Q_p$

\[\label{ext} 0 \rightarrow K \rightarrow V(x_L) \rightarrow K(-1) \rightarrow 0\]

where $K$ denotes the trivial $K$-representation and $K(-1)$ denotes the representation given by the inverse of the cyclotomic character.

Moreover, $V(x_L)$ is unramified outside $\l$ (in particular, $V(x_L)$ is locally a split extension at places dividing $p$). 
\end{corollary}
\begin{proof}The construction of $V(x_L)$ is a standard application of `Ribet's lemma' --- for example see \cite[Proposition 9.3]{BCBK} for an application in a more complicated situation. By construction we get an extension which is unramified outside $p\l$. But the extension $V(x_L)$ also splits at every place $\p|p$, since the local Galois representations at these places have an unramified character as a \emph{quotient} (\cite{Wiles}).
\end{proof}
 One can then ask whether $V(x_L)$ is actually ramified at $\l$ or not. We have constructed $V(x_L)$ using a family of Galois representations which is generically ramified at $\l$, and the monodromy filtrations on the local Weil-Deligne representations at classical points in this family are compatible with the extension structure (\ref{ext}), so this extension is, in contrast to what happens at primes above $p$, not obviously forced to split at $\l$. But in fact there are no extensions of $K(-1)$ by $K$ which are ramified at a single prime $\l$:
 \begin{lemma}
 The representation $V(x_L)$ of Corollary \ref{extn} is unramified everywhere.
 \end{lemma}
\begin{proof}First note that the twist $V(x_L)(1)$ corresponds to a line in the kernel  \[K^1(\l) = \ker [H^1_{\mathrm{ur},\{p\l\}}(G_{F},K(1)) \rightarrow \bigoplus_{\p|p} H^1(G_{F_{\p}},K(1))],\] where $H^1_{\mathrm{ur},\{p\l\}}(G_{F},K(1))$ denotes the global cohomology classes which are unramified away from $p\l$. If we denote by $K^1$ the kernel of the map \[H^1_{\mathrm{ur},\{p\}}(G_{F},K(1)) \rightarrow \bigoplus_{\p|p} H^1(G_{F_{\p}},K(1))\] from the Selmer group where the cohomology classes are unramified outside $p$, then there is an inclusion $K^1\subset K^1(\l)$. It remains to show that the inclusion $K^1\subset K^1(\l)$ is an equality. 

We just need to show that $K^1$ and $K^1(\l)$ have the same dimension. Kummer theory induces isomorphisms \[K^1(\l) \cong \ker[\OO_{F,p\l}^\times\otimes_{\Z}K \rightarrow \bigoplus_{\p|p} \widehat{F_{\p}^\times}\otimes_{\Z_p}K]=\ker[\alpha:\OO_{F,\l}^\times\otimes_{\Z}K \rightarrow \bigoplus_{\p|p} \widehat{F_{\p}^\times}\otimes_{\Z_p}K]\] and \[K^1 \cong \ker[\OO_{F,p}^\times\otimes_{\Z}K \rightarrow \bigoplus_{\p|p} \widehat{F_{\p}^\times} \otimes_{\Z_p}K]=\ker[\beta:\OO_{F}^\times\otimes_{\Z}K \rightarrow \bigoplus_{\p|p} \widehat{F_{\p}^\times} \otimes_{\Z_p}K].\] Here $\OO_{F,\mathfrak{a}}$ denotes the $\mathfrak{a}$-units, for an ideal $\mathfrak{a}$ of $\OO_F$ and $\widehat{F_{\p}^\times}$ denotes the $\Z_p$ module $\varprojlim_{n}F_{\p}^\times/(F_{\p}^\times)^{p^n}$. Observe that if $l$ is the rational prime under $\l$, there is an element $x$ of $\OO_{F,\l}^\times$ with $N_{F/\Q}(x) = l^a$ for a non-zero $a$ (take a suitable power of the ideal $\l$ to obtain a principal ideal and then take $x$ to be a generator, or the square of a generator). So the dimensions of the image and the domain of $\alpha$ are one larger than the dimensions of the image and the domain of $\beta$ (one sees that the image grows by considering composition with the product of norm maps to $\widehat{\Q_{p}^\times} \otimes_{\Z_p}K$, in which $x$ maps to something non-trivial, whilst $\OO_{F}^\times\otimes_{\Z}K$ maps to zero), hence the kernels have the same dimension.
\end{proof}
Note that if we allow ramification at \emph{two} auxiliary primes then one may obtain ramified extensions. However, the fact that the extension $V(x_L)$ turns out to be unramified at $\l$ indicates that it will be difficult to ensure that any extensions constructed using Ribet's method are ramified at these auxiliary primes.

\subsubsection{An example for $F=\Q$}\label{expeg}
We sketch how to construct an example of an ordinary $p$-adic modular form such that its transfer to a $p$-adic automorphic form on a definite quaternion algebra over $\Q$ satisfies the assumptions of Theorem \ref{raise}. All computations were done in Sage \cite{Sage}.

Note that the space of cusp forms $S_2(\Gamma_0(15))$ is one-dimensional. Denote by $f \in S_2(\Gamma_0(15))$ the normalised newform in this space. If we set $p=3$, then there is a unique $3$-adic Hida family passing through $f$, so we have an isomorphism of $\Lambda$-algebras $\mathbb{T}^0(\Gamma_0(15)) \cong \Lambda = \Z_p[[T]]$, where $\mathbb{T}^0(\Gamma_0(15))$ denotes the $\Lambda$-adic ordinary cuspidal Hida Hecke algebra (generated by Hecke operators prime to $15$). The form $f$ satisfies $T_{13}^2(f)-(13+1)^2S_{13}(f)=((-2)^2-(13+1)^2)f  = 0$ mod $3$.

Let $D$ be a definite quaternion algebra over $\Q$ which is non-split at the infinite place and the prime $5$, and split at all other primes. By the main result of \cite{ChenJL}, the generic fibre of the Hida family through $f$ transfers to a one--dimension rigid subspace of $\mathscr{E}^{D}(\Z)$.
\begin{proposition}
There is a characteristic $0$ point of the Hida family through $f$ such that the associated point $x$ of $\mathscr{E}^{D}(\Z)$ satisfies the conditions of Theorem \ref{raise} for $\l = (13)$. In particular there is a point of $\mathscr{E}^{D}((13))$ with the same system of Hecke eigenvalues as $x$ (away from $13$) which is a point of intersection between two irreducible components\footnote{One lying in $\mathscr{E}^{D,(13)\mhyphen old}((13))$, the other in $\mathscr{E}^{D,(13)\mhyphen new}((13))$.}.
\end{proposition}
\begin{proof}
If we consider the image of the Hecke operator $T_{13}^2-14^2S_{13}$ in $\Lambda$, we obtain an element $p^\mu P(T)U(T)$ where $\mu \in \Z_{\ge 0}$, $P(T)$ is a distinguished polynomial, and $U(T) \in \Lambda^\times$. Since $P(T)$ is distinguished, its roots lie in $p\OO_K$ for $K/\Q_p$ finite, so if the degree of $P(T)$ is non-zero then specialising our Hida family to the weights corresponding to the ideals $(T-\alpha) \subset \Lambda$ for $\alpha$ a root of $P(T)$ gives the desired $p$-adic modular form (which may be transferred to a definite quaternion algebra, non-split at $5$). If $P(T)$ has degree zero, then the $p$-adic valuation of $p^\mu U(\alpha)$ is equal to $\mu$ for all $\alpha \in p\Z_p$, so we just need to show that the valuation of the Hecke eigenvalue for $T_{13}^2-14^2S_{13}$ is not constant in our Hida family. This is easily done in this example: at weight $2$ the valuation is $1$, at weight $4$ it is $2$ (and at weight $10$ it is $3$).
\end{proof}

In fact, in the above example one can see directly that the conclusion of Theorem \ref{raise} holds, without applying the results of this paper (this was explained to the author by Frank Calegari) --- the new subspace of $S_2(\Gamma_0(13\cdot 15))$ is also one-dimensional, so we again have a Hida Hecke algebra isomorphic to $\Lambda$, and a similar argument to the above, applied to the ideal describing the intersection between the two Hida families, shows that is enough to observe that the $13$-new $\Lambda$-adic form and the $13$-old $\Lambda$-adic form become more congruent at weights other than $2$. This can be checked by doing the computation described in the above proof and applying \cite[Theorem 6.C]{Diamond}, or it can be checked directly.

\begin{remark}
By contrast, consider the (also one-dimensional) space $S_2(\Gamma_0(21))$ and its normalised newform $g$. If we set $p=3$ and $l=13$, then the valuation of the eigenvalue for $T_{13}^2-14^2S_{13}$ is $1$ for the weight $2j$ specialisations of the Hida family through $g$, with $1\le j \le 40$. So it is possible that the intersection between the two Hida families is just the ideal $(p)$ --- it would be interesting to determine whether this is the case  (for example, there may be no lift of the relevant Galois representation to a representation with coefficients in $\Z/p^2\Z$ whose local representation at $l$ has the correct form, in which case the intersection must be given by $(p)$).
\end{remark}
\begin{remark}
Finally, we note that the same arguments apply to the example discussed in \cite[Example 5.3.2]{EPW}. In this case, the two Hida families described in that example do intersect at a characteristic zero point.
\end{remark}

\section{Corrections to \cite{chicomp}}
We take this opportunity to make a couple of corrections to the paper \cite{chicomp}. 
\subsection{Proof of Lemma 7(ii)}
This proof should be modified as in the proof of Lemma \ref{normtrivial} above to take account of the change of variables $z \mapsto rz$ made when fixing an isomorphism $$\mathscr{A}_{x,r} \cong \prod_{\alpha=1}^n K'\langle T \rangle.$$ This entails computing the action of $\begin{pmatrix}
1 & -pr^{-1}\\0 & 1
\end{pmatrix}$ on $\mathscr{D}_{x,r}$. The statement of the lemma is unaffected. We reproduce the correct computation here (this should replace the displayed formulae at the bottom of \cite[pg. 346]{chicomp}):
\begin{align*}\langle e_{i,\alpha},f\cdot\begin{pmatrix}
1 & -pr^{-1}\\0 & 1
\end{pmatrix}\rangle &= \langle T^i, (\sum b_{j,\alpha} T^j)\cdot\begin{pmatrix}
1 & -pr^{-1}\\0 & 1
\end{pmatrix}\rangle &=\langle T^i\cdot\begin{pmatrix}
1 & pr^{-1}\\0 & 1\end{pmatrix}, \sum b_{j,\alpha} T^j \rangle\\
&= \langle (T+p)^i, \sum b_{j,\alpha} T^j \rangle\\ &= \langle \sum_{k=0}^i \binom{i}{k}p^{i-k}T^k,\sum b_{j,\alpha} T^j\rangle \\
&= \sum_{j=0}^i \binom{i}{j}p^{i-j}b_{j,\alpha}.\end{align*} 
We can now see that if we have $f=f\cdot \begin{pmatrix}
1 & -pr^{-1}\\0 & 1\end{pmatrix}$ we get $\sum_{j=0}^i \binom{i}{j}p^{i-j}b_{j,\alpha}=b_{i,\alpha}$ for all $i$ and $\alpha$, which implies that $f=0$. 
\subsection{Section 3}
A minor comment is that $\mathscr{E}^\mathrm{old}$ as defined in section 3.1 should be replaced by its underlying reduced space. More seriously, the statement in the proof of Theorem 15 that `$\mathfrak{M}$ is not Eisenstein since the Galois representation attached to $\phi$ is irreducible' is false. There are points on the cuspidal $\GL_2$ eigencurve with reducible attached Galois representation, and these show up on the eigencurve for the definite quaternion algebra. For example in the definite quaternion algebra case, in classical weight $2$, the space of constant functions on $D_f^\times$ always gives rise to such a point. Hence in the statement of Theorem 15 we must assume that the point $\phi$ does not give rise to a `very Eisenstein' (in the terminology of the current paper) maximal ideal in the Hecke algebra. The statement of Theorem 17 must also be modified to exclude these very Eisenstein points. We note that the question of level lowering/raising for points precisely of this kind has been studied in a recent preprint of Majumdar \cite{1307.4846}.

\section{Acknowledgments}
The author would like to thank C. Khare, who prompted the writing of this paper by asking if the results of \cite{chicomp} could be applied to the Eisenstein points on Hilbert modular eigenvarieties described in Section \ref{example}. The author is supported by Trinity College, Cambridge, and the Engineering and Physical Sciences Research Council.

\end{document}